\documentclass[DIV=14,letterpaper]{scrartcl}

\usepackage{tikz}
\usetikzlibrary{calc}
\usetikzlibrary{matrix}
\usepackage{amsmath,amssymb,amsfonts,amsthm}
\usepackage{graphicx}
\allowdisplaybreaks
\usepackage{subfigure,multicol,multirow}
\usepackage{makecell}
\usepackage{hyperref} 
\usepackage{diagbox}

\usepackage{bbm}

\theoremstyle{plain}% default

\newtheorem{pro}{\hspace{6mm}Proposition}[section]
\newtheorem{lem}{\hspace{6mm}Lemma}[section]

\theoremstyle{definition}

\theoremstyle{remark}

\newcommand{\V}[1]{\mathbf{#1}}

\newcommand{\email}[1]{\href{mailto:#1}{#1}}

\title{
Parameter-free inexact block Schur complement preconditioning for linear poroelasticity
under a hybrid Bernardi-Raugel and weak Galerkin finite element discretization
}
\author{
Weizhang Huang\thanks{Department of Mathematics, the University of Kansas, 1460 Jayhawk Blvd, Lawrence, KS 66045, USA (\email{whuang@ku.edu}).}
\and
Zhuoran Wang\thanks{Department of Mathematics, the University of Kansas, 1460 Jayhawk Blvd, Lawrence, KS 66045, USA (\email{wangzr@ku.edu}).}
}
\date{} 

\begin{document}

\maketitle

\textbf{Abstract.}
This work investigates inexact block Schur complement preconditioning for linear poroelasticity problems
discretized using a hybrid approach: Bernardi–Raugel elements for solid displacement and
lowest-order weak Galerkin elements for fluid pressure. When pure Dirichlet boundary conditions
are applied to the displacement, the leading block of the resulting algebraic system becomes almost singular
in the nearly incompressible (locking) regime, hindering efficient iterative solution.
To overcome this, the system is reformulated as a three-field problem with an inherent regularization
that maintains the original solution while ensuring nonsingularity. Analysis shows that
both the minimal residual (MINRES) and generalized minimal residual (GMRES) methods,
when preconditioned with inexact block diagonal and triangular Schur complement preconditioners,
achieve convergence independent of mesh size and the locking parameter for the regularized system.
Similar theoretical results are established for the situation with displacement subject to
mixed boundary conditions, even without regularization. Numerical experiments in 2D and 3D
confirm the benefits of regularization under pure Dirichlet conditions and
the robustness of the preconditioners with respect to mesh size and the locking parameter
in both boundary condition scenarios. Finally, a spinal cord simulation with discontinuous
material parameters further illustrates the effectiveness and robustness of the proposed iterative solvers.

\vspace{5pt}

\noindent
\textbf{Keywords:}
Poroelasticity, Schur complement preconditioning, Locking-free, Bernardi-Raugel elements, Inherent regularization

\vspace{5pt}

\noindent
\textbf{Mathematics Subject Classification (2020):}
65N30, 65F08, 65F10, 74F10

%%%%%%%%%%%%%%%%%%%%%%%%%%%%%%%%%%%%%%%%%%%%%%%%%%%%%%%%%%%%%%%
%%%%%%%%%%%%%%%%%%%%%%%%%%%%%%%%%%%%%%%%%%%%%%%%%%%%%%%%%%%%%%%
\section{Introduction}
\label{SEC:intro}

We consider the iterative solution of linear poroelasticity problems governed by
\begin{align}
\begin{cases}
  \displaystyle
  -\nabla \cdot \sigma(\V{u})
  + \alpha \nabla p
  = \mathbf{f}, \quad \text{ in } \Omega  \times (0,T],
\\
  \displaystyle
  \partial_t
  \big ( \alpha \nabla \cdot \V{u} +  c_0 p \big)
    - \nabla \cdot \left( \kappa \nabla p \right)
  = s, \quad \text{ in } \Omega  \times (0,T],
\end{cases}
\label{BVP-poro}
\end{align}
where
$\Omega \subset \mathbb{R}^d \, (d\geq 1)$ is a bounded and connected Lipschitz domain,
$T > 0$ denotes the final time,
$ \V{u} $ is the solid displacement,
$ \lambda={\nu E}/{\big((1-2\nu)(1+\nu)\big)} $ and $ \mu={E}/{\big(2(1+\nu)\big)} $  are the Lam\'{e} constants, $E$ is Young's modulus,
$\nu$ is Poisson's ratio,
$\sigma(\V{u}) = 2\mu \varepsilon(\V{u}) + \lambda (\nabla \cdot \V{u}) \mathbf{I}$
is the Cauchy stress for the solid,
$ \varepsilon(\V{u}) = \frac{1}{2} ( \nabla \V{u} + (\nabla \V{u})^T ) $ is the strain tensor, 
$\mathbf{I} $ is the identity operator,
$ \mathbf{f} $ is a body force,
$ p $ is the fluid pressure,
$ s $ is the fluid source,
$ \alpha $ (usually close to $1$) is the Biot-Willis constant accounting for the coupling of the solid and fluid,
$ c_0 > 0 $ is the constrained storage capacity,
and $ \kappa $ is the permeability constant.
We consider boundary conditions (BCs) for $u$ and $p$ in a general form as 
\begin{align}
\begin{cases}
&\V{u} = \V{u}_D, \quad \text{on } {\Gamma_{uD}}  \times (0,T], \\
& 
( \V{\sigma} - \alpha p \V{I} ) \cdot \V{n} = \V{t}_N, \quad \mbox{ on } \Gamma_{uN} \times (0,T] , \\
& p = p_D, \quad \text{on } {\Gamma_{pD}}  \times (0,T], \\
& \kappa \nabla p \cdot \V{n} = p_N, \quad \mbox{ on } \Gamma_{pN} \times (0,T],
\end{cases}
\label{BC-1}
\end{align}
where $\Gamma_{uN} \cup \Gamma_{uD} = \partial \Omega$ and $\Gamma_{pN} \cup \Gamma_{pD}
= \partial \Omega$, the Dirichlet and Neumann parts are non-overlapping, and
$\V{u}_D$, $\V{t}_N$, $p_D$, and $p_N$ are given functions.
Notice that the Neumann BC for the solid displacement on $\Gamma_{uN}$ has taken
the effects of the fluid pressure into consideration.
The initial conditions are specified as
\begin{align}
 \V{u} = \V{u}_0,
  \quad
  p = p_0,
  \quad
  \text{on} \; \Omega \times \{t=0\}.
  \label{IC-1}
\end{align}

The weak formulation of \eqref{BVP-poro} is 
to seek $(\V{u}(\cdot, t), p(\cdot, t)) \in (H^1(\Omega))^d\times H^1(\Omega)$, $0< t \le T$
satisfying $\V{u}|_{\Gamma_{uD}} = \V{u}_D$ and $p|_{\Gamma_{pD}} = \V{p}_D$ (in the weak sense) and
\begin{equation}
\begin{cases}
    2 \mu \Big( \varepsilon (\V{u}), \varepsilon (\V{v}) \Big)
      + \lambda ({\nabla \cdot \V{u}}, {\nabla \cdot \V{v}})
      - \alpha (p, {\nabla \cdot \V{v}})
      = (\mathbf{f}, \V{v}) + (\V{t}_N, \V{v})_{\Gamma_{uN}},
      \quad \forall \V{v}\in (H_{0}^1(\Omega))^d,
\\
    \displaystyle
    -\alpha (\nabla \cdot \V{u}_t, q)
    - c_0 \left( p_t, q \right)
      - \left( {\kappa} \nabla p, \nabla q \right)
    \displaystyle
    =  -\left( s, q \right) - (p_N,q)_{\Gamma_{pN}},
    \quad \forall q\in H_{0}^1(\Omega),
\end{cases}
  \label{poro_variationalform}
\end{equation}
where $(\cdot, \cdot)$,  $(\cdot, \cdot)_{\Gamma_{uN}}$, and $(\cdot, \cdot)_{\Gamma_{pN}}$ are the $L^2$ inner product over $\Omega$, $\Gamma_{uN}$, and $ \Gamma_{pN}$, respectively.
The test functions $\V{v} \in (H_{0}^1(\Omega))^d$ and $q \in H_0^1(\Omega) $
vanish on the Dirichlet boundary $\Gamma_{uD}$ and $\Gamma_{pD}$, respectively.

For the analysis of iterative solution, we need to distinguish two BC scenarios: $|\Gamma_{uN}| = 0$
(that will be referred to as the pure Dirichlet boundary condition (DBC) scenario) and $|\Gamma_{uN}| > 0$
(that will be referred to as the mixed BC scenario).
To explain this, we notice that the pressure related term $(p, \nabla\cdot\mathbf{v})$
in the first equation of (\ref{poro_variationalform}) can be expressed as
\begin{align}
  (p, \nabla\cdot\mathbf{v}) = - (\nabla p, \V{v}) + (p \V{n}, \V{v})_{\Gamma_{uN}}, \quad \V{v} \in (H_{0}^1(\Omega))^d.
\label{grad-1}  
\end{align}
For the pure DBC scenario, the second term on the right-hand side vanishes and
$(\bullet, \nabla\cdot\bullet)$ represents the gradient for the first argument, which in turn implies that
the operator is singular and for any solution $p$, $p+c$ is also a solution for any constant $c$.
On the other hand, for the mixed BC scenario the second term on the right-hand side of
(\ref{grad-1}) does not vanish in general and
the operator is not singular. 
The singularity of the operator will affect the performance and thus the choice of
iterative methods for solving the entire system.

Besides the singularity issue, the value of $\lambda$ is also a major challenge
for solving linear poroelasticity numerically.
As $\lambda \to \infty$,
we have $\nabla \cdot \V{u} \to 0$, which indicates the solid becomes nearly incompressible.
This causes the so-called locking phenomenon where numerical oscillations occur in the computed solution,
leading to deteriorate convergence order in discretization;
e.g. see \cite{phillipsWheeler_CompGeo_2009,Yi_SISC_2017}.
A variety of locking-free numerical methods have been developed
and applied to poroelasticity problems, including mixed finite element methods \cite{Ricardo2_SINUM_2016, RIVA6_SISC_2025}, hybridizable discontinuous Galerkin methods \cite{KRAUS4_CMAME_2021}, enriched Galerkin methods \cite{LeeYi_JSC_2023}, virtual element methods \cite{LiangRui_ApplNumMath_2024}, weak Galerkin finite element methods \cite{WangTavLiu_JCAM_2021}, discontinuous Galerkin methods \cite{ZhaoChungPark_IMA_2023}, among others.
Discretizing poroelasticity problems yields large and ill-conditioned linear systems that are needed to solve
at each time step. Tremendous effort has been made in solving those systems efficiently.
The developed methods include splitting methods (such as fixed-stress, fixed-strain, drained,
and undrained methods in 
\cite{AltmannDeiml_SISC_2025,Both5_ApplMathLett_2017,Hong4_MathModelMethApplSci_2020,RIVA6_SISC_2025})
that solve the solid and fluid sequentially at each time step
and block Schur complement preconditioning techniques (e.g. see
\cite{Hong-2023-MathComp,HuangWang_2025, Huang-Wang-2025-poro-reg,Lee-SISC-2017,Rodrigo6_SeMA_2024})
that improve convergence and efficiency of iterative solution of the whole system.

In this work, we investigate inexact block Schur complement preconditioning for linear poroelasticity problems.
We employ Bernardi–Raugel elements \cite{BernardiRaugel_MathComp_1985} for the displacement discretization and the lowest-order weak Galerkin elements for the pressure, coupled with implicit Euler time discretization.
The resulting hybrid scheme is known to be locking-free,
achieve optimal-order convergence for both displacement and pressure, and work for
both pure DBC and mixed BC scenarios; cf. \cite{WangTavLiu_JCAM_2021}.
When the displacement is subject to pure Dirichlet BCs,
the (1,2) and symmetric (2,1) blocks of the numerical system are singular.
Moreover, under the locking condition (i.e., $\lambda \to \infty$), the leading block corresponding to the solid component of the coupled system becomes nearly singular.
These features make the system difficult to solve efficiently.
To address this, we first reformulate the two-field system into
a three-field problem by introducing a numerical pressure variable and then add
an inherent regularization term to the third equation. 
This regularization preserves the original solution while ensuring nonsingularity of the new system.
Moreover, the eigenvalues of its Schur complement, preconditioned
by a simple approximation, remain bounded below and above by positive constants.
For the regularized system, we study the convergence of the minimal residual method (MINRES)
and the generalized minimal residual method (GMRES), preconditioned with inexact block diagonal
and triangular Schur complement preconditioners, respectively.
The residual bounds for MINRES and GMRES show that both iterative solvers converge independently of $h$ (the mesh size) and $\lambda$ (the locking parameter).

When the displacement is subject to mixed boundary conditions,
the (1,2) and (2,1) blocks of the two-field system are nonsingular.
As a result, the entire system is nonsingular even in the locking regime
and there is no need for regularization.
We consider the three-field formulation with the numerical pressure variable
and show that MINRES and GMRES, when preconditioned with inexact block Schur complement preconditioners,
also exhibit parameter-free convergence.

It is interesting to point out that two recent works,
\cite{HuangWang_2025} and \cite{Huang-Wang-2025-poro-reg},
are closely related to the current study. (Indeed, the current work
can be regarded as an extension of those two.) They all study
inexact block Schur complement preconditioning for linear poroelasticity
although there are significant differences between them.
\cite{HuangWang_2025} and \cite{Huang-Wang-2025-poro-reg} employ
the weak Galerkin discretization for both displacement and pressure,
which has been shown (e.g., see \cite{Wang2TavLiu_JCAM_2024})
to be locking-free and achieve optimal-order convergence only for the pure DBC scenario.
for the pure DBC scenario, it is still unclear if those features hold for the mixed BC scenario.
Moreover, no regularization is used in \cite{HuangWang_2025} and the corresponding
preconditioned system contains a small eigenvalue for the locking regime.
As a consequence, the convergence factor of preconditioned GMRES is parameter-free but not
for the asymptotic error constant. 
Parameter-free convergence is established in \cite{Huang-Wang-2025-poro-reg} using
an inherent regularization strategy.

A key motivation for this work is real-world applications including spinal cord simulation,
which typically involve mixed boundary conditions for solid displacement.
In this study, we adopt a different discretization approach: using Bernardi–Raugel elements
for the solid displacement and lowest-order weak Galerkin elements for the pressure.
This hybrid method is known to be locking-free and effective for both pure Dirichlet and
mixed displacement BCs; however, it has not yet been analyzed in the context of iterative solvers.
The goal of this work is to analyze block preconditioning and investigate the convergence behavior
of MINRES and GMRES for both types of boundary conditions, and to apply the results to
the spinal cord simulation.

The rest of paper is organized as follows.
In Section~\ref{SEC:formulation}, the discretization of the poroelasticity problem
(\ref{BVP-poro}) using hybrid Bernardi–Raugel and weak Galerkin elements
is described and its properties are discussed.
In Section~\ref{sec::DBC}, the inherent regularization strategy is discussed
for the linear poroelasticity problem subject to the pure DBC, and the convergence of MINRES and GMRES with inexact block Schur complement preconditioning is analyzed for the resulting regularized system.
Sections~\ref{sec::NBC} studies the convergence of MINRES and GMRES with block preconditioning for the system under mixed boundary conditions.
Numerical results for poroelasticity problems in both two and three dimensions are presented in Section~\ref{SEC:numerical}, confirming the parameter-free convergence of MINRES and GMRES, and illustrating the effectiveness of the proposed preconditioning strategies for a real-world application in spinal cord simulation
with discontinuous parameters.
Conclusions are drawn in Section~\ref{SEC:conclusions}.

%%%%%%%%%%%%%%%%%%%%%%%%%%

% %%%%%%%%%%%%%%%%%%%%%%%%%%%%%%%%%%%%%%%%%%%%%%%%%%%%%%%%%%%%%%%
\section{Discretization for poroelasticity}
\label{SEC:formulation}

In this section we describe the Bernardi–Raugel (BR) elements coupled with the lowest-order
WG elements for the discretization of the poroelasticity problem \eqref{BVP-poro}
and discuss their properties.

We first consider the discretization of \eqref{BVP-poro} subject to the boundary condition (\ref{BC-1}).
Assume that a quasi-uniform simplicial mesh $\mathcal{T}_h = \{K\}$  is given for $\Omega$,
where $h$ is the maximum element diameter.
We also assume that the mesh $\mathcal{T}_h$ is connected in the sense that any two of its elements
are connected by a chain of elements sharing interior facets.
For a generic element $K$, denote its vertices by $\V{x}_i$, $i = 0, ..., d$ and the facet facing $\V{x}_i$
by $e_i$ ($i = 0, ..., d$).
Then, the local $BR_1$ space on $K$ is defined as
\[
BR_1(K) = (P_1(K))^d + \text{span} \{ \V{n}_i \prod\limits_{\V{x}_j \in e_i} \lambda_j,
\; i = 0, ..., d\},
\]
where $\lambda_i$'s are barycentric coordinates/Lagrange-type linear basis functions,
$\V{n}_i$ is the unit outward normal to the $i$-th facet $e_i$,
and the product $\prod_{\V{x}_j \in e_i} \lambda_j$ takes all the linear basis functions
according to the vertices on $e_i$.
Having defined $BR_1$, we define the discrete weak function spaces as
\begin{align}
     \displaystyle
    \V{V}_h
     & = \big \{ \V{u}_h \in (H_0^1(\Omega))^d: \;
      \V{u}_h|_K \in BR_1(K), \;
      \forall K \in \mathcal{T}_h\big \} \cap \big \{ \V{u}_h|_{\Gamma_{uD}} = \V{u}_D \big \},
    \\ 
    \displaystyle
    W_h  & = \big \{p_h=\{p^\circ_h, p^\partial_h\}:\;
    p^\circ_h|_{K} \in P_0(K),\;
    p^\partial_h|_e \in P_0(e), \;
    \forall K \in \mathcal{T}_h,\;  e \in \partial K \big \}\cap \big \{ p_h |_{\Gamma_{pD}} = p_D \big \},
    \\
    \displaystyle
     \mathcal{P}_0 &= \big \{p_h=\{p^\circ_h\} :\; p^\circ_h|_K\in P_0(K),\; \forall K\in\mathcal{T}_h\big \} ,
     \label{P0}
\end{align}
where $P_0(K)$ and $P_0(e)$ denote the spaces of constant polynomials defined on element $K$ and facet $e$, respectively.
Notice that functions in $W_h$ consist of two parts, one defined in the interiors of the mesh elements
and the other on their facets.

Define the discrete weak gradient operator $\nabla_w: W_h \rightarrow RT_0(\mathcal{T}_h)$  for $p_h = (p_h^{\circ},p_h^{\partial})$ as
\begin{equation}
\label{weak-grad-1}
  (\nabla_w p_h, \mathbf{w})_K
  = (p^\partial_h, \mathbf{w} \cdot \mathbf{n})_{\partial K}
  - ( p^\circ_h , \nabla \cdot \mathbf{w})_K,
  \quad \forall \mathbf{w} \in RT_0(K),\quad \forall K \in \mathcal{T}_h ,
\end{equation}
where 
$\mathbf{n}$ is the unit outward normal to $\partial K$,
$(\cdot, \cdot)_K$ and $( \cdot, \cdot )_{\partial K}$ are the $L^2$ inner product on $K$ and $\partial K$, respectively,
and $RT_0(K)$ is the lowest-order Raviart-Thomas space defined as
\[
RT_0(K) = (P_0(K))^d + \mathbf{x} \, P_0(K).
\]
The analytical expression of $\nabla_w p_h$ can be obtained; see, e.g., \cite{HuangWang_CiCP_2015}.

For temporal discretization we consider a time partition of the interval $(0,T]$ given by  $0 = t_0 < t_1< ... <t_N = T$ and denote the time step as $\Delta t_n = t_n - t_{n-1}$. Using the implicit Euler scheme for temporal discretization and
the Bernardi-Raugel elements coupled with lowest-order WG for spatial discretization,
from (\ref{poro_variationalform}) we obtain the time marching scheme as seeking $\V{u}_h^n \in \V{V}_h$ and $p_h^n \in W_h$ such that 
% $\V{u}_h|_{\partial \Omega}= \V{u}_D^h$,
% $ p_h|_{\partial \Omega}=p_D^h $, where $\V{u}_D^h$ and $p_D^h$ are
% the $L^2$-projections of $\V{u}_D$ and $p_D$ on $\partial \Omega$,
\begin{equation}
\begin{cases}
\displaystyle
2\mu \sum_{K \in\mathcal{T}_h} (\varepsilon (\V{u}_h^n),\varepsilon (\V{v}_h))_K
+\lambda \sum_{K \in\mathcal{T}_h}(\overline{\nabla \cdot\V{u}_h^n}, \overline {\nabla \cdot\V{v}_h})_K
\\
\displaystyle
\qquad \qquad \qquad 
- \alpha \sum_{K \in\mathcal{T}_h} (p_h^{\circ,n}, \nabla \cdot\V{v}_h)_K
= \sum_{K \in\mathcal{T}_h}(\mathbf{f}^n,\V{v}_h )_K + \sum_{e \in \Gamma_{uN}} (\V{t}_N, \V{v})_{e} , \quad \forall \V{v}_h \in \V{V}_h^0,
\\
\displaystyle
-\alpha \sum_{K \in\mathcal{T}_h}  (\nabla \cdot\V{u}_h^n,q_h^{\circ})_K
- c_0 \sum_{K \in\mathcal{T}_h} (p_h^{\circ,n}, q_h^{\circ})_K
- \Delta t_n \sum_{K \in\mathcal{T}_h}(\kappa\nabla_w p_h^n,\nabla_wq_h)_K
\\
\displaystyle
\qquad \qquad \qquad 
= - \Delta t_n \sum_{K \in\mathcal{T}_h} (s^n,q_h^{\circ})_K
    -\alpha \sum_{K \in\mathcal{T}_h}  (\nabla \cdot\V{u}_h^{n-1},q_h^{\circ})_K
\\
\displaystyle
\qquad \qquad \qquad \qquad \qquad \qquad \qquad
- c_0 \sum_{K \in\mathcal{T}_h} (p_h^{\circ,n-1}, q_h^{\circ})_K - \sum_{e \in \Gamma_{pN}} (p_N, q_h^{\partial})_e
    ,\quad \forall q_h \in {W}_h^0 ,
\end{cases}
\label{EqnFullDisc1}
\end{equation}
where $\V{V}_h^0$ and $W_h^0$ are the homogeneous counterparts of $\V{V}_h$ and $W_h$, respectively,
with $\V{v}_h|_{\Gamma_{uD}} = \V{0}$ and $q_h|_{\Gamma_{pD}} = 0$, and
$\overline{\nabla \cdot \V{u}_h}$ denotes the elementwise $L^2$-projection of $\V{u}_h$
to the space of constants, i.e., $\overline{\nabla \cdot \V{u}_h}|_K \in P_0(K)$ for all
$K \in \mathcal{T}_h$. Notice that the use of $\overline{\nabla \cdot \V{u}_h}$
(and $\overline{\nabla \cdot \V{v}_h}$) for the dilation term of the solid in the first equation
is similar to the reduced integration technique \cite{MalkusHughes_1978} that is crucial for the locking-free
feature of the scheme.
The following lemma shows that the above discretization scheme is locking free and
achieves optimal-order convergence for both displacement and pressure.

\begin{lem}{\cite[Theorem 2]{WangTavLiu_JCAM_2021}}
\label{pro:poro_conv}
Let $ (\V{u}, p) $ and $(\V{u}_h,p_h) $ be the exact and numerical solutions of poroelasticity problem (\ref{BVP-poro}).
Under suitable regularity assumptions for the exact solution,
there holds 
\begin{align}
  \displaystyle
  \max_{1\leq n\leq N} \|\V{u}^n - \V{u}_h^n \|_{1}^2
   + \Delta t \sum_{n = 1}^N\| p^n-p_h^n \|^2 
  \displaystyle
  \leq C_1 h^2 + C_2 (\Delta t)^2,
  \label{pro:poro_conv-2}
\end{align}
where $\|\cdot \|_1$ and $\|\cdot \|$ denote the $H^1$ and $L^2$ norm, respectively,
and constants $C_1$ and $C_2$ depend on $\V{u}$ and $p$ but not on $h$, $\Delta t$, and $\lambda$.
\end{lem}

% \begin{proof}
%     The proof can be found in \cite[Theorem 2]{WangTavLiu_JCAM_2021}.
% \end{proof}

Omitting the superscript $n$,
we can rewrite the above system in a matrix-vector form  as 
\begin{equation}
    \begin{bmatrix}
        2 \mu A_1  + \lambda  A_0 & -\alpha B^T \\
       - \alpha B & -D
    \end{bmatrix}
    \begin{bmatrix}
        \V{u}_h \\
        \V{p}_h
    \end{bmatrix}
    =
    \begin{bmatrix}
        \mathbf{b}_1 \\
        \mathbf{b}_2
    \end{bmatrix},
    \label{2by2Scheme_matrix}
\end{equation}
where the matrices and right hand sides are given by
\begin{align}
& \V{v}_h^T A_1 \V{u}_h = \sum_{K \in\mathcal{T}_h} (\varepsilon (\V{u}_h),\varepsilon ( \V{v}_h))_K,
\quad \forall \V{u}_h, \V{v}_h \in \V{V}_{h}^0 ,
\label{A1-1}
\\
& \V{v}_h^T A_0 \V{u}_h = \sum_{K \in\mathcal{T}_h} (\overline {\nabla \cdot\V{u}_h},\overline {\nabla \cdot\V{v}_h})_K,
\quad \forall \V{u}_h, \V{v}_h \in \V{V}_{h}^0 ,
\label{A0-1}
\\
& \mathbf{q}_h^T B \V{u}_h =  \sum_{K \in \mathcal{T}_h}  (\nabla \cdot\V{u}_h,q_h^{\circ})_K,
\quad \forall \V{u}_h \in \V{V}_{h}^0 , \quad \forall q_h \in W_{h}^0,
\label{B-1}
\\
& \mathbf{q}_h^T D \mathbf{q}_h = c_0 \sum_{K \in \mathcal{T}_h} (p_h^{\circ}, q_h^{\circ})_K 
+ \Delta t \sum_{K \in\mathcal{T}_h}(\kappa\nabla_w p_h,\nabla_wq_h)_K,
\quad \forall p_h, q_h \in W_{h}^0 ,
\label{D-1}
\\
& \V{v}_h^T \V{b}_1 = \sum_{K \in\mathcal{T}_h} (\V{f}, \V{v}_h)_K + \sum_{e \in \Gamma_{uN}} (\V{t}_N, \V{v})_{e}, \quad \forall \V{v}_h \in \V{V}_{h}^0 ,
\label{b1-1}
\\
& \V{q}_h^T \V{b}_2 = - \Delta t \sum_{K \in\mathcal{T}_h} (s,q_h^{\circ})_K
    -\alpha \sum_{K \in\mathcal{T}_h}  (\nabla \cdot\V{u}_h,q_h^{\circ})_K
    \notag \\
    & \qquad \qquad 
    - c_0  \sum_{K \in\mathcal{T}_h} (p_h^{\circ}, q_h^{\circ})_K - \sum_{e \in \Gamma_{pN}} (p_N, q_h^{\partial})_e, \quad \forall q_h \in W_{h}^0.
\label{b2-1}
\end{align}
Here, $\V{u}_h$ and $\V{p}_h$ are used for both the discrete functions
and the corresponding vectors formed by their degrees of freedom.
In particular, for $\V{p}_h$, the degrees of freedom are ordered first with those associated
with element interiors and followed by those on element facets.

We now explore structures of the matrices.
The block $A_1$ is the stiffness matrix associated with the displacement and
is symmetric and positive definite (SPD).
The block B, defined as in \eqref{B-1}, can equivalently be written as
\begin{align}
    \mathbf{q}_h^T B \V{u}_h =  \sum_{K \in \mathcal{T}_h}  (\nabla \cdot\V{u}_h,q_h^{\circ})_K
    = \sum_{K \in \mathcal{T}_h}  (\overline{\nabla \cdot\V{u}_h},q_h^{\circ})_K
    ,
\quad \forall \V{u}_h \in \V{V}_h^0, \quad \forall q_h \in W_h^0.
\label{B-2}
\end{align}
Moreover, it has the structure 
\begin{equation}
B = \begin{bmatrix} B^{\circ} \\ 0 \end{bmatrix},
\label{B-3}
\end{equation}
where the number of rows of $B^{\circ}$ is equal to $N$ (the number of elements of $\mathcal{T}_h$) and
that of the zero block is equal to the number of the degrees of freedom of $q_h^{\partial}$
minus the number of boundary facets on $\Gamma_{pD}$.

From \eqref{A0-1} and (\ref{B-1}), it is not difficult to see 
\begin{align*}
\mathbf{u}_h^T A_0 \mathbf{v}_h 
&= \sum_{K\in\mathcal{T}_h} (\overline{\nabla \cdot\mathbf{v}_h},\overline{\nabla\cdot\mathbf{u}_h})_K
= \V{u}_h^T  (B^{\circ})^T (M_p^{\circ})^{-1} B^{\circ}  \V{v}_h,
\end{align*}
which implies
\begin{equation}
    A_0 = (B^{\circ})^T (M_p^{\circ})^{-1} B^{\circ} ,
\label{A0-2}
\end{equation}
where $M_p^{\circ} $ is the mass matrix that can be expressed as
\begin{align}
    \label{mass-1}
    M_p^{\circ} = \text{diag}(|K_1|, ..., |K_N|),
\end{align}
with $K_j$, $j = 1, ..., N$ denoting the elements of $\mathcal{T}_h$ and $|K_j|$ the volume of $K_j$.

The block $D$ in \eqref{D-1} has the structure as 
\begin{equation}
D = c_0 \begin{bmatrix}
    M_p^{\circ} & 0 \\[0.1in]
    0 & 0
\end{bmatrix}
+ \kappa \Delta t  A_p ,
\label{D-2}
\end{equation}
where $A_p$ is the stiffness matrix of the Laplacian operator for the pressure.

The right-hand side vector $\V{b}_2$ has the same row structure as $B$ (cf. (\ref{B-3})), i.e., 
\begin{align}
    \V{b}_2 = 
    \begin{bmatrix}
        \V{b}_2^{\circ}
        \\
        0
    \end{bmatrix} .
    \label{b2-2}
\end{align}

It is worth pointing out that the matrices $A_0$, $B$, and $D$ are the same
as the corresponding matrices resulting from the use of the weak Galerkin approximation for
both $\V{u}$ and $p$ (e.g., see \cite{Huang-Wang-2025-poro-reg}).
This is due to the use of the averaging for the second term in the first equation
of (\ref{EqnFullDisc1}) and the property (\ref{B-2}).
On the other hand, $A_1$ is different in these discretizations
since different approximation spaces are used for $\V{u}$
and $A_1$ corresponds to the strain operator $\epsilon(\cdot)$ instead
of the Laplacian as in \cite{Huang-Wang-2025-poro-reg}.

As stated in the introduction, the operator $(\bullet, \nabla\cdot\bullet)$
represents the gradient for the first argument and thus is singular 
for the pure DBC scenario (i.e., when the solid is subject to a Dirichlet
boundary condition over the entire boundary). However,
it does not represent the gradient for the mixed BC scenario (when the solid is
subject to a Neumann boundary condition on part of the boundary).
As shown in the following lemma, this property is preserved by
the discrete approximation of the operator, $B^{\circ}$.

\begin{lem}
\label{lem:B0-1}
Assume that the mesh $\mathcal{T}_h$ is connected in the sense that any two of its elements
are connected by a chain of elements sharing interior facets.
Then, $B^{\circ}$ is nonsingular for the mixed BC scenario and singular for the pure DBC
scenario. In the latter scenario, the null space of $(B^{\circ})^T$ is given by
\begin{align}
  \text{Null}((B^{\circ})^T) = \{ q_h \in \mathcal{P}_0(\mathcal{T}_h): \; q_{h,K} = C, \; \forall K \in \mathcal{T}_h; \; \text{ C is a constant} \} .
  \label{lem:B0-1-1}
\end{align}
\end{lem}

\begin{proof}
For any $q_h \in \text{Null}((B^{\circ})^T)$ and any $\V{v}_h \in \V{V}_h^0$,
\begin{align}
0 = \V{q}_h^T B^{\circ} \V{v}_h
= \sum_{K\in\mathcal{T}_h} 
({q}^{\circ}_h, \nabla \cdot \V{v}_h)_K
= \sum_{K\in\mathcal{T}_h} {q}_{h,K}
\int_K \nabla \cdot \V{v}_h
= \sum_{K\in\mathcal{T}_h} {q}_{h,K}
\int_{\partial K} \V{v}_h \cdot \V{n} .
\label{lem:B0-1-2}
\end{align}
Let $e$ be an arbitrary interior facet shared by two elements $K$ and $\tilde{K}$.
Choosing $\V{v}_h $ to be a facet bubble function that vanishes at the vertices and is nonzero only
on $e$, we obtain
\begin{align}
    (q_{h,K} - q_{h,\tilde{K}}) \int_e \V{v}_h \cdot \V{n} = 0,
\label{lem:B0-1-3}
\end{align}
which implies $q_{h,K} = q_{h,\tilde{K}}$. Since the mesh is assumed to be connected, we conclude
that
\begin{align}
q_{h,K} = C, \quad \forall K \in \mathcal{T}_h ,
\label{lem:B0-1-4}
\end{align}
for some constant $C$.

For the pure DBC scenario, $\V{v}_h$ vanishes on all boundary facets and (\ref{lem:B0-1-3})
holds only for interior facets. Thus, the functions in the null space of $(B^{\circ})^T$
can be expressed in the form (\ref{lem:B0-1-4}), which means $B^{\circ}$ is singular.

On the other hand, for the mixed BC scenario, we can consider a boundary facet $e$ on $\Gamma_{uN}$.
Similarly, from (\ref{lem:B0-1-2}) we can get $q_{h,K} \int_e \V{v}_h \cdot \V{n} = 0$, which, together with
(\ref{lem:B0-1-4}), implies that $q_h = 0$ on all elements.
Thus, we conclude that $B^{\circ}$ is nonsingular.
\end{proof}

Hereafter, we will frequently use the notation $\text{Null}\left((B^{\circ})^T\right)^{\perp}$.
From the above lemma, we see that for the mixed BC scenario, $\text{Null}\left((B^{\circ})^T\right)^{\perp} = \mathcal{P}_0$ (cf. \eqref{P0}) 
and for the pure DBC scenario, $\text{Null}\left((B^{\circ})^T\right)^{\perp} = \mathcal{P}_0 \cap \left\{ \int_{\Omega} p_h d\V{x} = 0\right\}$.

The following lemma shows that the discretization scheme \eqref{EqnFullDisc1}
satisfies the {inf-sup} condition. Notice that this cannot be induced from
the inf-sup condition of the pure WG scheme
(for both $\V{u}$ and $p$, e.g., see \cite{WangWangLiu_JSC_2023})
since $A_1$ is different in the current discretization.

\begin{lem}
\label{inf_sup}
There exists a constant $ 0<\beta<1 $ independent of $h$ such that 
\begin{align}
   \sup_{\V{v}_h\in \V{V}_{h}^0 ,\, \V{v}_h \neq \V{0}}
  \frac{\V{p}_h^T B^{\circ} \V{v}_h}
      {\Big(\V{v}_h^T A_1 \V{v}_h\Big)^{\frac{1}{2}}} 
    \geq \beta \; \Big(\V{p}_h^T M_p^{\circ} \V{p}_h\Big)^{\frac{1}{2}}, 
  \qquad \forall \V{p}_h \in \text{Null}\left((B^{\circ})^T\right)^{\perp}.
  \label{inf-sup-CGWG-2}
\end{align}
\end{lem}
\begin{proof}
   From \cite[Lemma 11.2.3]{BrennerScott_book}\footnote{Lemma 11.2.3 in \cite{BrennerScott_book} considers
   only two scenarios where $\tilde{\V{v}}$ satisfies only a Dirichlet BC or no BC over the entire boundary. It can be proven that the lemma extends to the situation where $\tilde{\V{v}}$ satisfies a Dirichlet BC
   only on part of the boundary.}, there exists a positive constant $\beta$ such that  
   for any $ \V{p}_h \in \text{Null}\left((B^{\circ})^T\right)^{\perp}$ (and thus $ p_h \in L^2(\Omega)$), there is a function $ \V{\tilde{v}} \in (H^1(\Omega))^d$ satisfying
    \begin{align}
        \frac{(\nabla \cdot \V{\tilde{v}}, p_h)}{\| \V{\tilde{v}} \|_1} \geq \beta \| p_h \|.
    \label{inf_sup-CGWG-3}    
    \end{align}
    From \cite[Hypothesis 2]{BernardiRaugel_MathComp_1985} (that was shown in
    \cite{BernardiRaugel_MathComp_1985} to hold for BR elements),
    there exists a global interpolation operator 
    $\V{P}_h:$ $(H^1(\Omega))^d \rightarrow \V{V}_h^0$ such that
    \begin{align}
    (\nabla \cdot \V{v}, p_h) = (\nabla \cdot (\V{P}_h\V{v}), p_h), \quad \forall \V{p}_h \in \text{Null}\left((B^{\circ})^T\right)^{\perp}, \;  \V{v} \in (H^1(\Omega))^d.
    \label{inf-sup-CGWG-1}
    \end{align}
    Taking $\V{v} = \V{P}_h \tilde{\V{v}}$, from \eqref{inf_sup-CGWG-3} and \eqref{inf-sup-CGWG-1} we get
\begin{align*}
    (\nabla \cdot (\V{P}_h \V{\tilde{v}}), p_h) 
    &= (\nabla \cdot \V{\tilde{\V{v}}}, p_h)
    \geq \beta \|\V{\tilde{\V{v}}} \|_1 \|  p_h \| 
    \geq  \beta \|\V{P}_h\V{\tilde{\V{v}}} \|_1 \|  p_h \|,
\end{align*}
which gives
\[
\frac{(\nabla \cdot (\V{P}_h \V{\tilde{v}}), p_h)}
{\|\V{P}_h\V{\tilde{\V{v}}} \|_1} \ge \beta \|  p_h \|.
\]
Since $\V{P}_h \V{\tilde{v}} \in \V{V}^0_h$, we have
\[
\sup\limits_{\V{v}_h \in \V{V}^0_h, \, \V{v}_h \neq \V{0}}
\frac{(\nabla \cdot \V{v}_h, p_h)}{\| \V{v}_h \|_1} \geq \beta \|p_h\| .
\]
This, together with the fact that
$(\V{v}_h^T A_1 \V{v}_h)^{\frac{1}{2}} \le \| \V{v}_h \|_1$,
gives \eqref{inf-sup-CGWG-2}.
\end{proof}

%%%%%%%%%%%%%%%%%%%%%%%%%%%%%%%%%%%%%%%%%%%%%%%%%%%%%%%%

\section{Convergence analysis of MINRES and GMRES with block preconditioning: Pure DBC scenario}
\label{sec::DBC}

In this section we study the convergence of MINRES and GMRES
with block preconditioning for the linear poroelasticity system (\ref{2by2Scheme_matrix})
under the pure Dirichlet boundary condition for the solid displacement.
The scenario with mixed BC will be studied in the next section.

With the pure DBC, the leading block of (\ref{2by2Scheme_matrix})
becomes nearly singular for the locking regime when the solid becomes nearly
incompressible (and thus $\lambda$ is large). To explain this, 
we can rewrite \eqref{2by2Scheme_matrix} as
\begin{equation}
    \begin{bmatrix}
        \epsilon A_1  +  A_0 & -\frac{\alpha \epsilon}{2\mu } B^T \\[0.05in]
        -\frac{\alpha \epsilon}{2\mu} B & -\frac{\epsilon}{2 \mu }D
    \end{bmatrix}
    \begin{bmatrix}
        \V{u}_h \\
        \V{p}_h
    \end{bmatrix}
    = \frac{\epsilon}{2 \mu}
    \begin{bmatrix}
        \mathbf{b}_1 \\
        \mathbf{b}_2
    \end{bmatrix},
    \label{2by2Scheme_matrix2}
\end{equation}
where $\epsilon = \frac{2\mu}{ \lambda }$.
From (\ref{A0-2}) and Lemma~\ref{lem:B0-1}, we can see that $A_0$ is singular
and therefore, the leading block, $\epsilon A_1 + A_0$ is nearly singular as
$\epsilon \to 0$. To avoid this difficulty in iterative solution, 
we introduce a numerical pressure variable $\V{z}_h = - (M_p^{\circ})^{-1} B^{\circ}\V{u}_h$
and rewrite the above system into a three-field system as
\begin{equation}
    \begin{bmatrix}
      A_1  &  B^T & -(B^{\circ})^T\\[0.1in]
     B & \displaystyle -\frac{2\mu}{ \alpha^2 }D & 0\\[0.1in]
     -B^{\circ} & 0& \displaystyle -\epsilon M_{P}^{\circ}
    \end{bmatrix}
    \begin{bmatrix}
        \mathbf{u}_h \\[0.1in]
        \frac{\alpha}{2\mu}\V{p}_h \\[0.1in]
        \frac{1}{\epsilon} \mathbf{z}_h
    \end{bmatrix}
    =
     \begin{bmatrix}
       \frac{1}{2\mu} \mathbf{b}_1 \\[0.1in]
       \frac{1}{\alpha} \mathbf{b}_2 \\[0.1in]
       0
    \end{bmatrix} .
    \label{3by3Scheme}
\end{equation}
Noticing that the (1,2) and (2,1) blocks are related to the (1,3) and (3,1) blocks,
we can further simplify the above system by making the (1,2) and (2,1) blocks vanishing. 
To this end, using (\ref{B-3}), (\ref{D-2}), and (\ref{b2-2}),
from the second block equation of the above system we can see that the degrees of
freedom of $\V{p}_h$ in the elements interiors and on facets,
$\V{p}_h^{\circ}$ and $\V{p}_h^{\partial}$, are related as
\begin{align}
\V{p}_h^{\partial} = - (A_p^{\partial\partial})^{-1} A_p^{\partial\circ} \V{p}_h^{\circ},
\label{popb-1}
\end{align}
where $A_p$ has been decomposed as
\[
A_p = \begin{bmatrix} A_p^{\circ \circ} & A_p^{\circ \partial}
\\ A_p^{ \partial \circ} & A_p^{ \partial \partial} \end{bmatrix} .
\]
Using (\ref{popb-1}) we can eliminate $\V{p}_h^{\partial}$ in (\ref{3by3Scheme})
and obtain
\begin{align}
    \begin{bmatrix}
      A_1  &  (B^{\circ})^T & -(B^{\circ})^T\\[0.1in]
     B^{\circ} & \displaystyle -\frac{2\mu}{ \alpha^2 }\tilde{D} & 0\\[0.1in]
     -B^{\circ} & 0& \displaystyle -\epsilon M_{p}^{\circ}
    \end{bmatrix}
    \begin{bmatrix}
        \mathbf{u}_h \\[0.1in]
        \frac{\alpha}{2\mu}\V{p}_h^{\circ} \\[0.1in]
        \frac{1}{\epsilon} \mathbf{z}_h
    \end{bmatrix}
    =
     \begin{bmatrix}
       \frac{1}{2\mu} \mathbf{b}_1 \\[0.1in]
       \frac{1}{\alpha} \mathbf{b}_2^{\circ} \\[0.1in]
       0
    \end{bmatrix} ,
    \label{3field_eqn1}
\end{align}
where 
\begin{equation}
    \label{D-3}
    \tilde{D} = c_0 M_p^{\circ} + \kappa \Delta t \left (A_p^{\circ\circ}-A_p^{\circ\partial} (A_p^{\partial\partial})^{-1}A_p^{\partial\circ} \right ) .
\end{equation}

Before further simplifying \eqref{3field_eqn1}, we want to regularize it since
it is nearly singular when $\epsilon$ is small.
Following \cite{Huang-Wang-2025-Stokes-reg,Huang-Wang-2025-poro-reg},
we consider an inherent regularization strategy here.
Define
\begin{align}
\label{e-1}
\V{1}= \frac{1}{\sqrt{N}}(1, ..., 1 )^T,
\end{align}
where $N$ is the number of elements of the mesh $\mathcal{T}_h$.
Recall that Lemma~\ref{lem:B0-1} implies $\V{1}^T B^{\circ} = 0$.
Then, multiplying the third equation of \eqref{3field_eqn1} from the left by $\V{1}^T$,
we get $(M_p^{\circ} \V{1})^T \V{z}_h = 0.$
Define
\begin{align}
    \label{w-1}
    \V{w} = \frac{M_p^{\circ} \V{1}}{\| M_p^{\circ} \V{1} \|}.
\end{align}
Then, for any positive parameter $\rho$, there holds
\begin{align}
\label{inh-1}
- \rho \V{w} \V{w}^T \V{z}_h = {0}.
\end{align}
Adding this to the third block equation of \eqref{3field_eqn1}, we obtain
\begin{align}
    \begin{bmatrix}
      A_1  &  (B^{\circ})^T & -(B^{\circ})^T\\[0.1in]
     B^{\circ} & \displaystyle -\frac{2\mu}{ \alpha^2 }\tilde{D} & 0\\[0.1in]
     -B^{\circ} & 0& \displaystyle -\epsilon M_{p}^{\circ} -\rho \V{w}\V{w}^T
    \end{bmatrix}
    \begin{bmatrix}
        \mathbf{u}_h \\[0.1in]
        \frac{\alpha}{2\mu}\V{p}_h^{\circ} \\[0.1in]
        \frac{1}{\epsilon} \mathbf{z}_h
    \end{bmatrix}
    =
     \begin{bmatrix}
       \frac{1}{2\mu} \mathbf{b}_1 \\[0.1in]
       \frac{1}{\alpha} \mathbf{b}_2^{\circ} \\[0.1in]
       0
    \end{bmatrix} .
    \label{3field_eqn2}
\end{align}
By adding the third row to the second row and re-arranging the unknown variables, we get
\begin{align}
     \begin{bmatrix}
      A_1  & 0 & -(B^{\circ})^T\\[0.1in]
     0 & \displaystyle -\frac{2\mu}{ \alpha^2 }\tilde{D} - \epsilon M_p^{\circ} - \rho \V{w}\V{w}^T & - \epsilon M_p^{\circ} - \rho \V{w}\V{w}^T\\[0.1in]
     -B^{\circ} & - \epsilon M_p^{\circ} - \rho \V{w}\V{w}^T& \displaystyle -\epsilon M_{p}^{\circ} -\rho \V{w}\V{w}^T
    \end{bmatrix}
     \begin{bmatrix}
        \mathbf{u}_h \\[0.1in]
        \frac{\alpha}{2\mu}\V{p}_h^{\circ} \\[0.1in]
        \frac{1}{\epsilon} \mathbf{z}_h - \frac{\alpha}{2\mu}\V{p}_h^{\circ}
    \end{bmatrix}
    =
     \begin{bmatrix}
       \frac{1}{2\mu} \mathbf{b}_1 \\[0.1in]
       \frac{1}{\alpha} \mathbf{b}_2^{\circ} \\[0.1in]
       0
    \end{bmatrix} ,
    \label{3field_eqn3}
\end{align}
where
\begin{align}
    \mathcal{\tilde{A}}_{3} =   \begin{bmatrix}
      A_1  & 0 & -(B^{\circ})^T\\[0.1in]
     0 & \displaystyle -\frac{2\mu}{ \alpha^2 }\tilde{D} - \epsilon M_p^{\circ} - \rho \V{w}\V{w}^T & - \epsilon M_p^{\circ} - \rho \V{w}\V{w}^T\\[0.1in]
     -B^{\circ} & - \epsilon M_p^{\circ} - \rho \V{w}\V{w}^T& \displaystyle -\epsilon M_{p}^{\circ} -\rho \V{w}\V{w}^T
    \end{bmatrix}.
    \label{A3_tilde}
\end{align}
This system is simpler than (\ref{3field_eqn1}) since the (1,2) and (2,1) blocks are zero.
It is worth pointing out that systems (\ref{3field_eqn2}) has the same solution
as the unregularized system (\ref{3field_eqn1}) since (\ref{inh-1}) is
an equality inherent to (\ref{3field_eqn1}).
The Schur complement for this system is
\begin{align}
     &\tilde{S}_3  = 
         \begin{bmatrix}
           \displaystyle  \frac{2\mu}{\alpha^2}\tilde{D}  +\epsilon 
         M_p^{\circ}  +
         \rho \V{w}\V{w}^T & \epsilon M_p^{\circ}  + \rho\V{w}\V{w}^T \\[0.1in]
     \displaystyle  \epsilon  M_p^{\circ}  + \rho\V{w}\V{w}^T   & \epsilon M_{p}^{\circ} + \rho \V{w}\V{w}^T + B^{\circ} A_1^{-1} (B^{\circ})^T
       \end{bmatrix} .
       \label{S3_tilde}
\end{align}
We take the approximation of $\tilde{S}_3$ as
\begin{align}
       & \hat{\tilde{S}}_3 = 
       \begin{bmatrix}
           \displaystyle  \frac{2\mu}{\alpha^2}\tilde{D}  +\epsilon 
         M_p^{\circ}  +
         \rho \V{w}\V{w}^T & 0 \\[0.1in]
     \displaystyle  0 &  M_{p}^{\circ} + \rho \V{w}\V{w}^T 
       \end{bmatrix} .
\label{hatS3_tilde}
\end{align} 
Moreover, we assume that $\rho$ is chosen
as $\rho = \mathcal{O}(h^d)$; see (\ref{rho-1}) for a specific choice.
This assumption is needed to avoid small eigenvalues of
$\hat{\tilde{S}}_3^{-1} \tilde{S}_3$ (see Lemma~\ref{lem:eigen_bound_3field_tilde} below).

\begin{lem}
    \label{lem:b-bound}
    If $\rho$ is taken as
    \begin{align}
        \label{rho-1}
        \rho = \frac{\beta^2 \lambda_{\max}(M_p^\circ) \lambda_{\min}(M_p^\circ)}
        {\lambda_{\max}(M_p^\circ) + \gamma^2 \lambda_{\min}(M_p^\circ)},
    \end{align}
    then the eigenvalues of $(M_p^\circ + \rho \V{w}\V{w}^T)^{-1} ( \epsilon M_p^{\circ} + \rho \V{w}\V{w}^T + B^{\circ} A_1^{-1} (B^{\circ})^T )$ lie in the interval 
    \begin{align}
            \Big[ C_1 + \epsilon + \mathcal{O}(N \rho^2), \; C_2 + \epsilon \Big],
            \label{lem:b-bound-1}
    \end{align}
     where
    \begin{align}
    \label{C4C5}
        C_1 = \frac{\beta^2 \gamma^2  \lambda_{\min}(M_p^\circ)}
        {\lambda_{\max}(M_p^\circ) + \gamma^2 \lambda_{\min}(M_p^\circ)},
         \quad
         C_2  = C_{\text{Korn}} + \frac{\beta^2  \lambda_{\max}(M_p^\circ)}
        {\lambda_{\max}(M_p^\circ) + \gamma^2 \lambda_{\min}(M_p^\circ)},
    \end{align}
    where $C_{\text{Korn}}$ is the constant from the Korn inequality.
\end{lem}

\begin{proof}
The proof is similar to that of Lemma~5.1 of \cite{Huang-Wang-2025-poro-reg}.
The interested reader is referred to the reference for detail.
A key of the proof is the inf-sup condition, which holds for both the current
discretization (cf. Lemma~\ref{inf_sup}) and the pure WG discretization
in \cite{{Huang-Wang-2025-poro-reg}} although $A_1$ is different in these situations.
\end{proof}

\begin{lem}
\label{lem:eigen_bound_3field_tilde}
The eigenvalues of $\hat{\tilde{S_3}}^{-1} \tilde{S_3}$ lie in
\begin{align}
   \Bigg[ \frac{C_1}{1+C_1} \cdot  \frac{\frac{2\mu}{\alpha^2} \lambda_{\min} (\tilde{D})}{\rho 
   + \frac{2\mu}{\alpha^2} \lambda_{\min} (\tilde{D})} + \mathcal{O}(\epsilon) + \mathcal{O}(N \rho^2), \quad 1+ C_2 +\epsilon\Bigg],
    \label{lem:eigen_bound_3field_tilde-1}
\end{align}
where $C_1$ and $C_2$ are defined in (\ref{C4C5}).
\end{lem}

\begin{proof}
The proof is similar to that of \cite[Lemma~5.2]{Huang-Wang-2025-poro-reg}.
\end{proof}

% \begin{lem}
% \label{lem:eigen_bound_3field}
% The eigenvalues of $\hat{S}_3^{-1} {S_3}$ lie in
% \begin{align}
%    \Bigg[ \frac{C_1}{1+C_1} \cdot  \frac{\frac{2\mu}{\alpha^2} \lambda_{\min} (\tilde{D})}{\rho + \epsilon \lambda_{\max}(M_p^{\circ})
%    + \frac{2\mu}{\alpha^2} \lambda_{\min} (\tilde{D})} + \mathcal{O}(\epsilon) + \mathcal{O}(N \rho^2), \quad 1+ C_2 +\epsilon\Bigg],
%     \label{lem:eigen_bound_3field-1}
% \end{align}
% where $C_1$ and $C_2$ are defined in (\ref{C4C5}).
% \end{lem}

% \begin{proof}
% Result can be derived from Lemma~\ref{lem:eigen_bound_3field_tilde}.
% \end{proof}

It is known that for quasi-uniform meshes,
$\lambda_{\min}(M_p^{\circ}) = \mathcal{O}(h^d)$,
$\lambda_{\max}(M_p^{\circ}) = \mathcal{O}(h^d)$, 
and ${\lambda_{\min}(\tilde{D})}
= c_0\mathcal{O}(h^d) + \kappa\Delta t \mathcal{O}(h^d)$ (e.g., see \cite{KamenskiHuangXu_2014}).
Then, we have $\rho = \mathcal{O}(h^d)$, $\mathcal{O}(N \rho^2) = \mathcal{O}(h^d)$, and
$C_1$ and $C_2$ are constants.
As a consequence, Lemmas~\ref{lem:b-bound} and \ref{lem:eigen_bound_3field_tilde}  imply that
the eigenvalues of
$(M_p^\circ + \rho \V{w}\V{w}^T )^{-1} ( \epsilon M_p^{\circ} + \rho \V{w}\V{w}^T + B^{\circ} A_1^{-1} (B^{\circ})^T )$ and $\hat{\tilde{S}}_3^{-1}{\tilde{S}_3}$ are bounded 
below and above essentially by positive constants independent of $h$ and $\epsilon$.
Moreover, it can be shown that it is not necessary to choose $\rho$
exactly as in (\ref{rho-1}) and Lemmas~\ref{lem:b-bound} and \ref{lem:eigen_bound_3field_tilde}
still hold when $\rho$ is taken as
\begin{align}
\rho < \beta^2 \lambda_{\min}(M_p^{\circ}), \quad 
\rho \sim \lambda_{\min}(M_p^{\circ}).
\label{rho-2}
\end{align}
In our computation, we take $\rho = 0.1 \,\lambda_{\min}(M_p^{\circ})$.

We now consider the iterative solution of
system (\ref{3field_eqn3}) using MINRES with a block diagonal
preconditioner ${\mathcal{\tilde{P}}}_{d}$ and GMRES with a block triangular preconditioner
${\mathcal{\tilde{P}}}_{t}$, where
\begin{align}
    & {\mathcal{\tilde{P}}}_{d} = 
    \begin{bmatrix}
        A_1 & 0 \\ 0 & \hat{\tilde{S}}_3 
    \end{bmatrix}
    = 
    \begin{bmatrix}
        A_1 & 0 & 0\\
        0 & \frac{2\mu}{\alpha^2}\tilde{D}  +\epsilon 
         M_p^{\circ}  +
         \rho \V{w}\V{w}^T & 0 \\
   0 & 0 & M_{p}^{\circ} + \rho \V{w} \V{w}^T
    \end{bmatrix},
    \label{3field-diag-0}
\\
& {\mathcal{\tilde{P}}}_{t} =
    \begin{bmatrix}
        A_1 & 0 \\ \begin{bmatrix} 0 \\ - B^{\circ} \end{bmatrix}
        & - \hat{\tilde{S}}_3 
    \end{bmatrix}
    = 
    \begin{bmatrix}
        A_1 & 0 & 0\\
        0 &  - \frac{2\mu}{\alpha^2}\tilde{D}  -\epsilon 
         M_p^{\circ}  -
         \rho \V{w}\V{w}^T & 0 \\
    -B^{\circ} & 0 & -M_{p}^{\circ} -  \rho \V{w} \V{w}^T
    \end{bmatrix} .
    \label{3field-tri-0}
\end{align}

\begin{lem}
\label{lem:eigen_bound_diag_poro}
The eigenvalues of $ \mathcal{\tilde{P}}_{d}^{-1} \mathcal{\tilde{A}}_{3} $ lie in
$[-a_1,-b_1] \cup [c_1,d_1]$, where
\begin{align*}
  & a_1 = \frac{1}{2} \big (\sqrt{4(1+C_2)+1}+1\big ) + \mathcal{O}(\epsilon) ,
  \qquad
   b_1 =  \frac{1}{2} \big (\sqrt{ \frac{4 C_1}{1+C_1} \cdot \frac{\frac{2\mu}{\alpha^2} \lambda_{\min} (\tilde{D})}{\rho 
   + \frac{2\mu}{\alpha^2} \lambda_{\min} (\tilde{D})} + 1} - 1\big ) + \mathcal{O}(\epsilon) ,
   \\
   & c_1 = \frac{1}{2} \big (\sqrt{ \frac{4 C_1}{1+C_1} \cdot \frac{\frac{2\mu}{\alpha^2} \lambda_{\min} (\tilde{D})}{\rho 
   + \frac{2\mu}{\alpha^2} \lambda_{\min} (\tilde{D})} + 1} - 1\big ) + \mathcal{O}(\epsilon),  
    \qquad
   d_1 = \frac{1}{2} \big (\sqrt{4(1+C_2)+1}+1\big ) + \mathcal{O}(\epsilon),
%    \label{eigen_bound_diag}
\end{align*}
and $C_1$ and $C_2$ are defined in (\ref{C4C5}).
\end{lem}

\begin{proof}
The eigenvalue problem of the preconditioned system $\mathcal{\tilde{P}}^{-1}_{d}\mathcal{\tilde{A}}_{3}$  reads as
   \begin{align*}
    \begin{bmatrix}
        A_1 & -C^T \\[0.05in]
        -C & -\tilde{S}_{3} + 
        \begin{pmatrix}
            0 & 0 \\ 0 & B^{\circ} A_1^{-1} (B^{\circ})^T
        \end{pmatrix}
        \end{bmatrix}
    \begin{bmatrix}
       \V{u}_h \\
        \V{w}_h
    \end{bmatrix} = \lambda
   \begin{bmatrix}
        A_1 & 0 \\
        0 & \hat{\tilde{S}}_{3}
    \end{bmatrix} 
        \begin{bmatrix}
      \V{u}_h \\
        \V{w}_h
    \end{bmatrix} ,
\end{align*}
where $ C = \begin{bmatrix}
    0 \\ B^{\circ}
\end{bmatrix}.$ 
It leads to the system
\begin{align*}
\begin{cases}
      A_1 \V{u}_h - C^T \V{w}_h = \lambda A_1 \V{u}_h \\
      -C\V{u}_h - \tilde{S}_{3} \V{w}_h  + 
        \begin{pmatrix}
            0 & 0 \\ 0 & B^{\circ} A_1^{-1} (B^{\circ})^T
        \end{pmatrix}\V{w}_h = \lambda \hat{\tilde{S}}_{3} \V{w}_h.
\end{cases}
\end{align*}
When $\lambda \neq 1$, solving the first equation gives $\V{u}_h = \frac{-1}{\lambda - 1} A_1^{-1} C^T \V{w}_h$.
Plugging this into the second equation, we have
\begin{align*}
    \lambda^2 \hat{\tilde{S}}_{3} \V{w}_h - \lambda \Big(\hat{\tilde{S}}_{3} - \tilde{S}_{3} + 
        \begin{pmatrix}
            0 & 0 \\ 0 & B^{\circ} A_1^{-1} (B^{\circ})^T
        \end{pmatrix} \Big) \V{w}_h - \tilde{S}_{3} \V{w}_h = \V{0},
\end{align*}
which yields
\begin{align*}
        \lambda^2 - \lambda \frac{\V{w}_h^T\Big(\hat{\tilde{S}}_{3} - \tilde{S}_{3} + \begin{pmatrix}
            0 & 0 \\ 0 & B^{\circ} A_1^{-1} (B^{\circ})^T
        \end{pmatrix}\Big) \V{w}_h}{\V{w}_h^T \hat{\tilde{S}}_{3} \V{w}_h} - \frac{\V{w}_h^T \tilde{S}_{3} \V{w}_h}{\V{w}_h^T\hat{\tilde{S}}_{3} \V{w}_h} = 0.
\end{align*}
Denoting 
\begin{align*}
    & a = \frac{\V{w}_h^T\Big(\hat{\tilde{S}}_{3} - \tilde{S}_{3} +   \begin{pmatrix}
            0 & 0 \\ 0 & B^{\circ} A_1^{-1} (B^{\circ})^T
        \end{pmatrix} \Big) \V{w}_h}{\V{w}_h^T \hat{\tilde{S}}_{3} \V{w}_h}, \\
        & b = \frac{\V{w}_h^T \tilde{S}_{3} \V{w}_h}{\V{w}_h^T\hat{\tilde{S}}_{3} \V{w}_h},
\end{align*}
we have 
\begin{align*}
    \lambda^2 - a\lambda - b = 0,
\end{align*}
which has the roots
\begin{align*}
    \lambda_{+} = \frac{a + \sqrt{a^2 + 4b}}{2},\quad
    \lambda_{-} = \frac{a - \sqrt{a^2 + 4b}}{2}.
\end{align*}
From Lemma~\ref{lem:eigen_bound_3field_tilde}, 
\[
\frac{C_1}{1+C_1} \cdot  \frac{\frac{2\mu}{\alpha^2} \lambda_{\min} (\tilde{D})}{\rho
   + \frac{2\mu}{\alpha^2} \lambda_{\min} (\tilde{D})} + \mathcal{O}(\epsilon) + \mathcal{O}(N \rho^2) \leq b \leq  1+ C_2 +\epsilon.
\]
It is not difficult to show that
\[
-1 \leq a \leq 1.
\]
We can obtain the conclusion from these bounds on $a$ and $b$
and the expressions of $\lambda_{\pm}$.
\end{proof}

\begin{pro}
\label{pro:MINRES_conv-DBC}
The residual of MINRES applied to a preconditioned system associated
with the coefficient matrix $\tilde{\mathcal{P}}_{d}^{-1} \tilde{\mathcal{A}}_{3} $ is bounded by   
\begin{align}
    \frac{\| \V{r}_{2k}\| }{\| \V{r}_0 \|} 
    % \; {\stackrel{<}{\sim}} \; 
    \leq 2 \left( \frac{\sqrt{\frac{a_1d_1}{b_1 c_1}} - 1}{\sqrt{\frac{a_1d_1}{b_1 c_1}} + 1}\right)^k,
     \label{pro:MINRES_conv-DBC-1}
\end{align}
where $a_1$, $b_1$, $c_1$, and $d_1$ are given in Lemma~\ref{lem:eigen_bound_diag_poro}.
\end{pro}

\begin{proof}
The result follows from Lemma~\ref{lem:eigen_bound_diag_poro} and \cite[Theorem 6.13]{Elman-2014}. 
\end{proof}

Notice that $a_1$, $b_1$, $c_1$, and $d_1$ are positive.
Moreover, they are constant essentially when a quasi-uniform mesh is used.
Then, the above proposition implies that the convergence of MINRES is essentially
independent of $\epsilon$ and $h$.

\begin{pro}
\label{pro:gmres_conv_DBC}
The residual of GMRES applied to the preconditioned system $ \tilde{\mathcal{P}}_{t}^{-1} \tilde{\mathcal{A}}_{3} $ is bounded by
\begin{align}
\frac{\| \V{r}_k\|}{\| \V{r}_0\|} 
% \; {\stackrel{<}{\sim}} \; 
\le
&
2\Bigg(2 + \epsilon + C_{2} + \sqrt{C_{\text{Korn}} \frac{\lambda_{\max} (M_p^{\circ})}{\lambda_{\min} (A_1)}} \Bigg) 
\notag \\
& \times \left( \frac{ \sqrt{1+C_2} - \sqrt{ \frac{C_1}{1+C_1} \cdot  \frac{\frac{2\mu}{\alpha^2} \lambda_{\min} (\tilde{D})}{\rho 
   + \frac{2\mu}{\alpha^2} \lambda_{\min} (\tilde{D})}}     }{  \sqrt{ 1+C_2 }   + \sqrt{\frac{C_1}{1+C_1} \cdot  \frac{\frac{2\mu}{\alpha^2} \lambda_{\min} (\tilde{D})}{\rho 
   + \frac{2\mu}{\alpha^2} \lambda_{\min} (\tilde{D})}}  } + \mathcal{O}(\epsilon) + \mathcal{O}(N \rho^2)  \right)^{k-1}  .
\label{pro:gmres_conv_DBC-0}
\end{align}
\end{pro}

\begin{proof}
From \cite[Lemma~A.1]{HuangWang_CiCP_2025}, we know that the residual of GMRES for the preconditioned system $\tilde{\mathcal{P}}_{t}^{-1} \tilde{\mathcal{A}}_{3} $ is bounded as 
    \begin{align}
        \frac{\| \V{r}_k \|}{\|\V{r}_0\|} \le
        (1+\|A_1^{-1}(B^{\circ})^T\| + \| \hat{\tilde{S}}_{3}^{-1} \tilde{S}_{3} \|) \min\limits_{\substack{p \in \mathbb{P}_{k-1}\\ p(0) = 1}} \| p(\hat{\tilde{S}}_{3}^{-1} \tilde{S}_{3}) \| .
        \label{eigen_bound_tri_poro-DBC-1}
    \end{align}
Lemma~\ref{lem:eigen_bound_3field_tilde} implies that
\[
\| \hat{\tilde{S}}_{3}^{-1} \tilde{S}_{3} \| \leq 1 +\epsilon + C_2.
\]
Moreover, using the Korn inequality and the fact that both $A_1$ and $M_p^{\circ}$ are SPD, we have
\begin{align*}
  \| A_1^{-1} (B^{\circ})^T \|^2
    & \leq \sup_{\V{p}_h \neq 0} \frac{\V{p}_h^T B^{\circ} A_1^{-1} A_1^{-1} (B^{\circ})^T \V{p}_h}{\V{p}_h^T  \V{p}_h} \nonumber
    \\
   & = \sup_{\V{p}_h \neq 0} \frac{\V{p}_h^T (M_p^{\circ})^{\frac{1}{2}} (M_p^{\circ})^{-\frac{1}{2}}B^{\circ} A_1^{-1}A_1^{-1} (B^{\circ})^T (M_p^{\circ})^{-\frac{1}{2}} (M_p^{\circ})^{\frac{1}{2}} \V{p}_h}{\V{p}_h^T \V{p}_h}
   \notag \\
   & \le \lambda_{\max} (A_1^{-1}) \lambda_{\max} (M_p^{\circ})
 \sup_{\V{u}_h \neq 0} \frac{\V{u}_h^T (B^{\circ})^T (M_p^{\circ})^{-1} B^{\circ} \V{u}_h}{\V{u}_h^T A_1 \V{u}_h}
 \notag \\
   & = \lambda_{\max} (A_1^{-1}) \lambda_{\max} (M_p^{\circ})
 \sup_{\V{u}_h \neq 0} \frac{\V{u}_h^T A_0 \V{u}_h}{\V{u}_h^T A_1 \V{u}_h}
 \notag \\
 & \le C_{\text{Korn}} \frac{\lambda_{\max} (M_p^{\circ})}{\lambda_{\min} (A_1)} .
\end{align*}
For the minmax problem in \eqref{eigen_bound_tri_poro-DBC-1}, by shifted Chebyshev polynomials
(e.g., see \cite[Pages 50-52]{Greenbaum-1997}) and Lemma~\ref{lem:eigen_bound_3field_tilde}
(where we denote the lower and upper bounds by $a$ and $b$, respectively), we have
\begin{align*}
   \min\limits_{\substack{p \in \mathbb{P}_{k-1}\\ p(0) = 1}} \| p( \hat{\tilde{S}}_{3}^{-1} \tilde{S}_{3} ) \|
& = \min\limits_{\substack{p \in \mathbb{P}_{k-1}\\ p(0) = 1}} \max_{i=1,..., N} |p(\lambda_i( \hat{\tilde{S}}_{3}^{-1} \tilde{S}_{3} )|
\le \min\limits_{\substack{p \in \mathbb{P}_{k-1}\\ p(0) = 1}} \max_{\gamma \in \left[a,b\right]} |p(\gamma)| 
\\
& \leq 2 \left(\frac{\sqrt{b} - \sqrt{a}}{\sqrt{b} + \sqrt{a}} \right)^{k-1} ,
\end{align*}
which leads to the result in \eqref{pro:gmres_conv_DBC-0}.
\end{proof}

The above proposition shows that the convergence of GMRES is essentially independent of $\epsilon$ and $h$.

To conclude this section we comment that the above convergence analysis for MINRES and GMRES has been
carried out for the reduced system (\ref{3field_eqn3}) that does not include $\V{p}_h^{\partial}$.
Unfortunately, this system is not convenient for implementation since the coefficient
matrix contains $(A_p^{\partial \partial})^{-1}$ in $\tilde{D}$.
To avoid this difficulty, we can add $\V{p}_h^{\partial}$ back to (\ref{3field_eqn3})
using (\ref{popb-1}). Then, we obtain
\begin{equation}
    \mathcal{A}_3
    \begin{bmatrix}
        \mathbf{u}_h \\[0.05in]
        \frac{\alpha}{2\mu}\V{p}_h \\[0.05in]
        \frac{1}{\epsilon} \mathbf{w_h} - \frac{\alpha}{2\mu}\V{p}_h^{\circ}
    \end{bmatrix}
    =
     \begin{bmatrix}
       \frac{1}{2\mu} \mathbf{b}_1 \\[0.05in]
       \frac{1}{\alpha} \mathbf{b}_2 \\[0.05in]
       0
    \end{bmatrix} ,
    \label{3field_eqn4}
\end{equation}
where
\begin{align*}
        \qquad 
    \mathcal{A}_3 = \begin{bmatrix}
      A_1  &  0 & -(B^{\circ})^T\\[0.1in]
     \displaystyle  0 &-\frac{2\mu}{\alpha^2}D - \epsilon \begin{pmatrix}
         M_p^{\circ} & 0 \\
         0 & 0
     \end{pmatrix}- \begin{pmatrix}
         \rho \V{w}\V{w}^T & 0 \\
         0 & 0
     \end{pmatrix}
     & -\epsilon \begin{pmatrix} M_p^{\circ}\\ 0 \end{pmatrix} - \begin{pmatrix}\rho \V{w}\V{w}^T\\ 0 \end{pmatrix} \\[0.1in]
     \displaystyle  -B^{\circ} & -\epsilon \begin{pmatrix} M_p^{\circ} & 0 \end{pmatrix} - \begin{pmatrix} \rho\V{w}\V{w}^T & 0 \end{pmatrix} & -\epsilon M_{p}^{\circ} - \rho \V{w}\V{w}^T
    \end{bmatrix} .
\end{align*}
The Schur complement for this system is
\begin{align}
     &S_3  = 
         \begin{bmatrix}
           \displaystyle  \frac{2\mu}{\alpha^2}D  +\epsilon \begin{pmatrix}
         M_p^{\circ} & 0 \\
         0 & 0
     \end{pmatrix}+ \begin{pmatrix}
         \rho \V{w}\V{w}^T & 0 \\
         0 & 0
     \end{pmatrix} & \epsilon \begin{pmatrix} M_p^{\circ} \\ 0 \end{pmatrix} + \begin{pmatrix} \rho\V{w}\V{w}^T \\ 0 \end{pmatrix} \\[0.1in]
     \displaystyle  \epsilon \begin{pmatrix} M_p^{\circ} & 0 \end{pmatrix} + \begin{pmatrix} \rho\V{w}\V{w}^T & 0 \end{pmatrix}  & \epsilon M_{p}^{\circ} + \rho \V{w}\V{w}^T + B^{\circ} A_1^{-1} (B^{\circ})^T
       \end{bmatrix} .
       \label{S3}
\end{align}
The approximate Schur complement and block diagonal and triangular preconditioners
corresponding to (\ref{hatS3_tilde}), (\ref{3field-diag-0}), and (\ref{3field-tri-0})
are given by
\begin{align}
       & \hat{S}_3 = 
       \begin{bmatrix}
           \displaystyle  \frac{2\mu}{\alpha^2}D  +\epsilon \begin{pmatrix}
         M_p^{\circ} & 0 \\
         0 & 0
     \end{pmatrix}+ \begin{pmatrix}
         \rho \V{w}\V{w}^T & 0 \\
         0 & 0
     \end{pmatrix} & 0 \\[0.1in]
     \displaystyle  0&  M_{p}^{\circ} + \rho \V{w}\V{w}^T
       \end{bmatrix} ,
\label{hatS3}
\\
    & {\mathcal{{P}}}_{d} = 
    \begin{bmatrix}
        A_1 & 0 \\ 0 & \hat{{S}}_3 
    \end{bmatrix}
    = 
    \begin{bmatrix}
        A_1 & 0 & 0\\
        0 & \frac{2\mu}{\alpha^2}{D}  +
        \begin{pmatrix}
            \epsilon 
         M_p^{\circ} & 0 \\
         0 & 0
        \end{pmatrix}     
          + \begin{pmatrix}
             \rho \V{w}\V{w}^T & 0 \\0 & 0 
          \end{pmatrix}
          & 0 \\
   0 & 0 & M_{p}^{\circ} + \rho \V{w}\V{w}^T
    \end{bmatrix},
    \label{3field-diag-full}
\\
& {\mathcal{{P}}}_{t} =
    \begin{bmatrix}
        A_1 & 0 \\ \begin{bmatrix} 0 \\ - B^{\circ} \end{bmatrix}
        & - \hat{{S}}_3 
    \end{bmatrix}
    = 
    \begin{bmatrix}
        A_1 & 0 & 0\\
        0 &  - \frac{2\mu}{\alpha^2}{D} -
        \begin{pmatrix}
            \epsilon 
         M_p^{\circ} & 0 \\
         0 & 0
        \end{pmatrix}     
          -\begin{pmatrix}
             \rho \V{w}\V{w}^T & 0 \\0 & 0 
          \end{pmatrix}& 0 \\
    -B^{\circ} & 0 & -M_{p}^{\circ} - \rho \V{w}\V{w}^T
    \end{bmatrix} .
    \label{3field-tri-full}
\end{align}
Using the results in Propositions~\ref{pro:MINRES_conv-DBC} and \ref{pro:gmres_conv_DBC} we can show that
the convergence of MINRES and GMRES applied to $\mathcal{P}_d^{-1} \mathcal{A}_3$ and 
$\mathcal{P}_t^{-1} \mathcal{A}_3$, respectively, is also independent of $h$ and $\epsilon$ (or $\lambda$).

%%%%%%%%%%%%%%%%%%%%%%%%%%%%%%%%%%%%%%%%%%%%%%%%%%%%%%%%%%%%%%%
\section{Convergence analysis of MINRES and GMRES with block preconditioning: Mixed BC scenario}
\label{sec::NBC}

In this section we study the convergence of MINRES and GMRES
with block preconditioning for the linear poroelasticity system (\ref{2by2Scheme_matrix})
under mixed boundary conditions for the solid displacement. 

Recall that for the mixed BC scenario, $(B^{\circ})^T$ has full column rank
(cf. Lemma~\ref{lem:B0-1}) and $B^{\circ} A_1^{-1} (B^{\circ})^T,$ which appears
in the Schur complement, is nonsingular. As a consequence, the entire system
is nonsingular too and its iterative solution does not need regularization
even for the locking regime when $\lambda$ is large.

In principle, we can use either the two-field formulation (\ref{2by2Scheme_matrix})
or the three-field formulation (\ref{3by3Scheme}). For notational consistency
with the pure DBC scenario and coding convenience, we consider the three-field formulation
here for the mixed BC scenario. We start the analysis with the reduced system 
\eqref{3field_eqn3} (without regularization), i.e., 
\begin{align}   
    \begin{bmatrix}
        A_1  &  0 & -(B^{\circ})^T\\[0.1in]
     0 & \displaystyle -\frac{2\mu}{ \alpha^2 }\tilde{D}  - \epsilon M_p^{\circ}& - \epsilon M_p^{\circ}\\[0.1in]
     -B^{\circ} &  -\epsilon M_p^{\circ} & \displaystyle -\epsilon M_{P}^{\circ}
    \end{bmatrix}
      \begin{bmatrix}
        \mathbf{u}_h \\[0.1in]
        \frac{\alpha}{2\mu}\V{p}_h^{\circ} \\[0.1in]
        \frac{1}{\epsilon} \mathbf{z}_h - \frac{\alpha}{2\mu}\V{p}_h^{\circ}
    \end{bmatrix}
    =
     \begin{bmatrix}
       \frac{1}{2\mu} \mathbf{b}_1 \\[0.1in]
       \frac{1}{\alpha} \mathbf{b}_2^{\circ} \\[0.1in]
       0
    \end{bmatrix} .
    \label{NBC_reduced_system}
\end{align}
Denote
\begin{align}
    \tilde{\mathcal{A}}_{3,N} =       
    \begin{bmatrix}
        A_1  &  0 & -(B^{\circ})^T\\[0.1in]
     0 & \displaystyle -\frac{2\mu}{ \alpha^2 }\tilde{D}  - \epsilon M_p^{\circ}& - \epsilon M_p^{\circ}\\[0.1in]
     -B^{\circ} &  -\epsilon M_p^{\circ} & \displaystyle -\epsilon M_{P}^{\circ}
    \end{bmatrix} .
    \label{A-NBC-3}
\end{align}
The Schur complement, its approximation, and block diagonal and triangular preconditioners
for this system will be denoted as $\tilde{S}_{N}$, $\hat{\tilde{S}}_{N}$,
$\tilde{{\mathcal{P}}}_{d,N}$, and $\tilde{{\mathcal{P}}}_{t,N}$, respectively.
They correspond to (\ref{S3_tilde}), (\ref{hatS3_tilde}), (\ref{3field-diag-0}),
and (\ref{3field-tri-0}) without regularization.
The estimation for the eigenvalues for the corresponding matrices and convergence rates for MINRES
and GMRES is similar to the pure DBC scenario in the previous section except that there is no need
to deal with a small eigenvalue as $\epsilon \to 0$ for the current situation.
For completeness, we list the relevant results without providing proofs.

\begin{lem}
\label{lem:eigen_bound_3field_NBC}
The eigenvalues of $\hat{\tilde{S}}_{N}^{-1} \tilde{S}_{N}$ lie in
\begin{align}
   \Bigg[ \frac{\beta^2}{1+\beta^2} + \mathcal{O}(\epsilon), 1 +\epsilon + C_{\text{Korn}}\Bigg],
    \label{lem:eigen_bound_3field_NBC-1}
\end{align}
where $C_{\text{Korn}}$ is the constant from the Korn inequality and $\beta$ is the constant from
the inf-sup condition.
\end{lem}

\begin{lem}
\label{lem:eigen_bound_diag_poro-NBC}
The eigenvalues of $ \tilde{\mathcal{P}}_{d,N}^{-1} \tilde{\mathcal{A}}_{3,N} $ lie in $[-a_2,-b_2]\cup [c_2,d_2]$,
where
\begin{align*}
&a_2 = \frac{1}{2} \big (\sqrt{4(1+C_{\text{Korn}}) + 1} + 1\big ) + \mathcal{O}(\epsilon),\qquad
b_2 = \frac{1}{2} \big (\sqrt{4 \frac{\beta^2}{1+\beta^2} + 1} - 1\big ) + \mathcal{O}(\epsilon),
\\
& c_2 = \frac{1}{2} \big (\sqrt{4 \frac{\beta^2}{1+\beta^2} + 1} - 1\big ) + \mathcal{O}(\epsilon) ,
\qquad
d_2 = \frac{1}{2} \big (\sqrt{4(1+C_{\text{Korn}}) + 1} + 1\big ) + \mathcal{O}(\epsilon) .
\end{align*}
\end{lem}

\begin{pro}
\label{pro:MINRES_conv-NBC}
The residual of MINRES applied to a preconditioned system associated
with the coefficient matrix $\tilde{\mathcal{P}}_{d,N}^{-1} \tilde{\mathcal{A}}_{3,N} $ is bounded by   
\begin{align}
    \frac{\| \V{r}_{2k}\| }{\| \V{r}_0 \|} 
    \le 2  \left(\frac{\sqrt{\frac{a_2 d_2}{b_2 c_2}}-1}{\sqrt{\frac{a_2 d_2}{b_2 c_2}}+1} \right)^k  .
     \label{pro:MINRES_conv-NBC-1}
\end{align}
\end{pro}

\begin{pro}
\label{pro:gmres_tri_poro-NBC}
The residual of GMRES applied to the preconditioned system $ \tilde{\mathcal{P}}_{t,N}^{-1} \tilde{\mathcal{A}}_{3,N} $ is bounded by
\begin{align}
\frac{\| \V{r}_k\|}{\| \V{r}_0\|} 
% \; {\stackrel{<}{\sim}} \; 
\le
&
2\Bigg(2 + \epsilon + C_2 + \sqrt{C_{\text{Korn}} \frac{\lambda_{\max} (M_p^{\circ})}{\lambda_{\min} (A_1)}} \Bigg) 
\left(\frac{ \sqrt{1+C_{\text{Korn}}} - \sqrt{\frac{\beta^2}{1+\beta^2}} }{  \sqrt{1+C_{\text{Korn}}} + \sqrt{\frac{\beta^2}{1+\beta^2}}} +\mathcal{O}(\epsilon) \right)^{k-1}  .
\label{gmres_tri_poro-NBC-0}
\end{align}
\end{pro}

The above propositions show that the convergence of MINRES and GMRES with inexact block Schur complement preconditioning
is essentially independent of $h$ and $\epsilon$ (or $\lambda$) for the linear poroelasticity with mixed BCs
where a quasi-uniform mesh is used.

Like the pure DBC scenario in the previous section, the preconditioning and MINRES and GMRES are actually implemented
for the full three-field system (see (\ref{3field_eqn4}), (\ref{3field-diag-full}), and (\ref{3field-tri-full})
with $\rho = 0$) and MINRES and GMRES have similar parameter-free convergence.

%%%%%%%%%%%%%%%%%%%%%%%%%%%%%%%%%%%%%%%%%%%%%%%%%%%%%%%%%%%%%%%

%%%%%%%%%%%%%%%%%%%%%%%%%%%%%%%%%%%%%%%%%%%%%%%%%%%%%%%%%%%%%%%
\section{Numerical experiments}
\label{SEC:numerical}

In this section, we present numerical results in both two and three dimensions
to illustrate the effectiveness of MINRES and GMRES, using
the block diagonal and triangular Schur complement preconditioners, respectively.
We examine scenarios with both pure Dirichlet and mixed boundary conditions.
Additionally, a spinal cord simulation is included to demonstrate
the robustness of the proposed preconditioning strategies in a real-world application.

For the linear poroelasticity problem with pure Dirichlet boundary conditions, we take
the regularization constant as $\rho = 0.1\, \lambda_{\min}(M_p^{\circ})$ (cf. \eqref{rho-2}). 
We use MATLAB's function {\em minres} and {\em gmres} with $tol = 10^{-8}$,
a maximum of 1000 iterations, and the zero vector as the initial guess for both 2D and 3D examples.
For {\em gmres}, we set the restart parameter is 30.
For the action of inversion of block preconditioners $\mathcal{P}_{d}$ and $ \mathcal{P}_{t}$,
linear systems associated with the leading block $A_1$ \eqref{A1-1} are solved
using the conjugate gradient (CG) method, preconditioned
with an incomplete Cholesky decomposition of $A_1$ computed 
using MATLAB's function {\em ichol} with threshold dropping.
For the action of inversion of the middle block, we use the CG method
preconditioned with an incomplete Cholesky decomposition of
\[
\frac{2\mu}{\alpha^2}D +\epsilon \begin{pmatrix}
         M_p^{\circ} & 0 \\ 0 & 0
    \end{pmatrix} .
\]
The action of inversion of the last block is carried out using matrix-vector multiplication as follows.
Using the Sherman-Morrison-Woodbury formula, we have
\begin{align*}
    \big ( M_p^{\circ} + \rho \V{w} \V{w}\big )^{-1}
    = (M_p^{\circ})^{-1} - \frac{\rho (M_p^{\circ})^{-1} \V{w} \V{w}^T
    (M_p^{\circ})^{-1}}{1+ \rho \V{w}^T (M_p^{\circ})^{-1} \V{w}} .
\end{align*}
Multiplication of this matrix with vectors is straightforward and can be computed efficiently since
$M_p^{\circ}$ is diagonal. Notice that for the mixed BC scenario, no regularization is needed
so we set $\rho = 0$. The setting of the solvers, including tolerance, maximum number
of iterations, initial guess, and preconditioner inversion methods, is the same for
both the pure DBC and mixed BC scenarios.

%%%%%%%%%%%%%%%%%%%%%%%%%%%%%%%%%%%%%%%%%%%%%%%%%%%%%
\subsection{Smooth solutions}
\label{Section_LinPoro2by2}

We first consider examples with smooth solutions.
The first example is a two-dimensional poroelasticity problem taken from \cite{SISCLeePierMarRog2019}.
Its right-hand side functions are given by
\begin{align*}
\mathbf{f} & = 
 - t \begin{bmatrix}
        -8 \pi^2 \mu \cos(2 \pi x) \sin(2 \pi y) - \frac{2 \pi^2 \mu}{\lambda + \mu} \sin(\pi x) \sin(\pi y) \\
      + 4 \pi^2 \mu \sin(2\pi y) 
      + \pi^2 \cos(\pi x+\pi y)
      + \alpha \pi \cos(\pi x)\sin(\pi y) 
      \\[0.08in]
        8 \pi^2 \mu \sin(2 \pi x) \cos(2 \pi y) - \frac{2 \pi^2 \mu}{\lambda + \mu} \sin(\pi x) \sin(\pi y) \\
      - 4 \pi^2 \mu \sin(2\pi x) 
      + \pi^2 \cos(\pi x+\pi y)
      + \alpha \pi \sin(\pi x)\cos(\pi y)
\end{bmatrix},
\\ %[0.08in]
s & = -c_0 \sin(\pi x)\sin(\pi y) 
+ \frac{\pi \alpha}{\lambda + \mu} \sin(\pi x+\pi y)
	  +t \Big(2 \pi^2 \sin(\pi x) \sin(\pi y)\Big) .
\end{align*}
We take the values of the parameters as $c_0 = 1$, $\kappa = 1$, $\mu = 1$,
$\lambda = 1 \text{ or } 10^4$, and $\Delta t = 10^{-3}$ and $ 10^{-6}$.
We consider two boundary scenarios,  one with pure Dirichlet boundary conditions for both $\V{u}$ and $p$, and the other with mixed boundary conditions.
\begin{description}
% \begin{itemize}
\item[Scenario I.]
Both $\V{u}$ and $p$ have pure Dirichlet boundary conditions: 
\[
\V{u} = \V{u}_D \text{ on } {\partial \Omega},
\quad p = p_D \text{ on } {\partial \Omega}.
\]
\item[Scenario II.]
Both $\V{u}$ and $p$ have mixed boundary conditions:
\begin{align*}
& 
\V{u} = \V{u}_D \text{ on } \partial \Omega \setminus \{ x = 1 \}, \qquad  ( \V{\sigma} - \alpha p \V{I} ) \cdot \V{n} = \V{t}_N \text{ on } x = 1; 
\\
&
p = p_D   \text{ on } \partial \Omega \setminus \{ x = 1 \}, \qquad
 p_N = \kappa \nabla p \cdot \V{n}  \text{ on } x = 1.
\end{align*}
% \end{itemize}
\end{description}
The boundary data $\V{u}_D$, $p_{D}$, $\V{t}_N$, and $p_N$ are chosen such that
the exact solution is given by
\begin{align*}
    \V{u} = t 
    \begin{bmatrix}
        (-1 + \cos(2 \pi x)) \sin(2 \pi y)  + \frac{1}{\lambda + \mu} \sin(\pi x) \sin(\pi y) \\[0.08in]
        \sin(2 \pi x) (1 - \cos(2 \pi y)) + \frac{1}{\lambda + \mu} \sin(\pi x) \sin(\pi y)     
    \end{bmatrix},
    \quad
    p =  -t  \sin(\pi x) \sin(\pi y).
\end{align*}

The second example is a three-dimensional test problem with the right-hand side functions as
\begin{align*}
\mathbf{f} & = t
 \begin{bmatrix}
 4\mu \cos(2\pi x) \sin(2\pi y) \sin(2\pi z) \pi^2 + \frac{(4\mu + \lambda)}{(\mu + \lambda)} \sin(\pi x) \sin(\pi y) \sin(\pi z) \pi^2 \\
- \cos(\pi x) \cos(\pi y) \sin(\pi z) \pi^2 - \cos(\pi x) \sin(\pi y) \cos(\pi z) \pi^2 \\
+ 8\pi^2 \mu \left(-1 + \cos(2\pi x)\right) \sin(2\pi y) \sin(2\pi z) + \alpha \pi \cos(\pi x) \sin(\pi y) \sin(\pi z)
\\[0.08in]
- \pi^2 \cos(\pi x) \cos(\pi y) \sin(\pi z) + \frac{(4\mu + \lambda)}{(\mu + \lambda)} \sin(\pi x) \sin(\pi y) \sin(\pi z) \pi^2 \\
- \sin(\pi x) \cos(\pi y) \cos(\pi z) \pi^2 + 16\pi^2 \mu \sin(2\pi x) \left(1 - \cos(2\pi y)\right) \sin(2\pi z) \\
- 8\pi^2 \mu \sin(2\pi x) \cos(2\pi y) \sin(2\pi z) + \alpha \pi \sin(\pi x) \cos(\pi y) \sin(\pi z)
\\[0.08in]
 \frac{(4\mu + \lambda)}{(\mu + \lambda)} \sin(\pi x) \sin(\pi y) \sin(\pi z) \pi^2 + 4\pi^2 \mu \sin(2\pi x) \sin(2\pi y) \cos(2\pi z) \\
- \cos(\pi x) \sin(\pi y) \cos(\pi z) \pi^2 - \sin(\pi x) \cos(\pi y) \cos(\pi z) \pi^2 \\
+ 8\pi^2 \mu \left(-1 + \cos(2\pi z)\right) \sin(2\pi x) \sin(2\pi y) + \alpha \pi \sin(\pi x) \sin(\pi y) \cos(\pi z)
\end{bmatrix},
\\
s &= \frac{\alpha \pi}{\mu + \lambda} \Big( \cos(\pi x) \sin(\pi y) \sin(\pi z) + \sin(\pi x) \cos(\pi y) \sin(\pi z)
\\
& \qquad \qquad + \sin(\pi x) \sin(\pi y) \cos(\pi z) \Big)
 + (3 \pi^2 t + c_0)\sin(\pi x) \sin(\pi y) \sin(\pi z) .
\end{align*}
The exact solution is
\begin{align*}
   & \V{u} = t 
    \begin{bmatrix}
        (-1 + \cos(2 \pi x)) \sin(2 \pi y)  \sin(2 \pi z)  + \frac{1}{\lambda + \mu} \sin(\pi x) \sin(\pi y)\sin(\pi z) \\[0.08in]
        2\sin(2 \pi x) (1 - \cos(2 \pi y))\sin(2 \pi z)  + \frac{1}{\lambda + \mu} \sin(\pi x) \sin(\pi y) \sin(\pi z)    \\[0.08in]
        \sin(2 \pi x)\sin(2 \pi y) (-1 + \cos(2 \pi z))  + \frac{1}{\lambda + \mu} \sin(\pi x) \sin(\pi y) \sin(\pi z)    
    \end{bmatrix},
    \\
    &     p =  t  \sin(\pi x) \sin(\pi y) \sin(\pi z) .
\end{align*}
The BCs and other parameters are chosen the same as for the two-dimensional example.

As stated in Propositions~\ref{pro:MINRES_conv-DBC} and~\ref{pro:gmres_conv_DBC} for systems
with pure Dirichlet boundary conditions and Propositions~\ref{pro:MINRES_conv-NBC}
and~\ref{pro:gmres_tri_poro-NBC} for systems with mixed boundary conditions,
the convergence of MINRES and GMRES is essentially independent of $h$ and $\epsilon$ (or $\lambda$).
Tables~\ref{Poro-2D} and~\ref{Poro-3D} list the number of iterations for the 2D and 3D examples under
two types of boundary conditions.
We observe that the iteration numbers remain consistent with variations in $\Delta t$, $\lambda$, and $h$. 
Moreover, the number of MINRES iterations is approximately twice that of GMRES,
consistent with the theoretical findings in Propositions~\ref{pro:MINRES_conv-DBC},
\ref{pro:gmres_conv_DBC}, \ref{pro:MINRES_conv-NBC},
and~\ref{pro:gmres_tri_poro-NBC} where similar bounds are given for 
the ratios $\|\V{r}_{2k}\|/\|\V{r}_0\|$ for MINRES and $\|\V{r}_{k}\|/\|\V{r}_0\|$ for GMRES.

\begin{table}[htb!]
    \centering
    \caption{The 2D Example for linear poroelasticity: The number of MINRES and GMRES iterations required to reach convergence for preconditioned systems, with $\lambda = 1$, $10^{4}$ and
    $\Delta t = 10^{-3}$, $10^{-6}$.}
           \begin{tabular}{|c|c|c|c|c|c|c|c|}
        \hline
        & & & \multicolumn{4}{c|}{$N$} \\ \cline{4-7}
       & $\Delta t$ & $\lambda$ & 918 & 3680 & 14720 & 58608 \\ \hline
       \multicolumn{7}{|c|}{Scenario I (Pure Dirichlet boundary condition)}
       \\ \hline
       MINRES & $10^{-3}$ & $1$ & 37  & 40 & 43 &   44 \\ 
        & $10^{-3}$ & $10^{4}$ &  30 &  33& 35 &  36  \\
         & $10^{-6}$ & $1$ &  26  & 28 & 30 &   35 \\ 
        & $10^{-6}$ & $10^{4}$ &  30 &32  & 34 & 36  \\
       \hline 
       GMRES & $10^{-3}$ & $1$ & 23  & 23 & 23 &23   \\ 
        & $10^{-3}$ & $10^{4}$ & 15  & 16 &16  &  16  \\ 
        & $10^{-6}$ & $1$ &  18 & 19 & 19 & 20 \\ 
        & $10^{-6}$ & $10^{4}$ &  15  & 16 &16  &  16 \\ \hline
       \multicolumn{7}{|c|}{Scenario II (Mixed boundary condition)}
       \\ \hline
              MINRES & $10^{-3}$ & $1$ & 47  & 49 & 52 &  52  \\ 
        & $10^{-3}$ & $10^{4}$ & 40  &  40& 40 &  40  \\
         & $10^{-6}$ & $1$ & 38  & 40 & 40 &  45  \\ 
        & $10^{-6}$ & $10^{4}$ & 40  & 40 & 40 &  40 \\
       \hline 
       GMRES & $10^{-3}$ & $1$ & 25  & 26 & 26 & 26  \\ 
        & $10^{-3}$ & $10^{4}$ & 21  & 22 & 22 &23    \\ 
        & $10^{-6}$ & $1$ &  23 & 23 & 23 &  25\\ 
        & $10^{-6}$ & $10^{4}$ &  21 & 22 & 22 & 23  \\ \hline
    \end{tabular}
    \label{Poro-2D}
\end{table}

\begin{table}[htb!]
    \centering
    \caption{The 3D Example for linear poroelasticity: The number of MINRES and GMRES iterations required to reach convergence for preconditioned systems, with $\lambda = 1$, $10^{4}$ and
    $\Delta t = 10^{-3}$, $10^{-6}$.}
           \begin{tabular}{|c|c|c|c|c|c|c|c|}
        \hline
        & & & \multicolumn{4}{c|}{$N$} \\ \cline{4-7}
       & $\Delta t$ & $\lambda$  & 65171 & 526031 & 1777571 & 4217332\\ \hline
       \multicolumn{7}{|c|}{Scenario I (Pure Dirichlet boundary condition)}
       \\ \hline
       MINRES & $10^{-3}$ & $1$   &30  & 31 & 34   & 34 \\ 
        & $10^{-3}$ & $10^{4}$   &  52& 52 &  52  & 54\\
         & $10^{-6}$ & $1$   & 30 & 31 &  31   & 31\\ 
        & $10^{-6}$ & $10^{4}$  &  52& 52 &  51  &52\\
       \hline 
       GMRES & $10^{-3}$ & $1$  & 20 &20  & 20  & 21\\ 
        & $10^{-3}$ & $10^{4}$  &  28&  27& 27   &27 \\ 
        & $10^{-6}$ & $1$  & 20 &20  & 20  & 20\\ 
        & $10^{-6}$ & $10^{4}$  &  28&  27& 27    &27\\ \hline
       \multicolumn{7}{|c|}{Scenario II (Mixed boundary condition)}
       \\ \hline
      MINRES & $10^{-3}$ & $1$  & 40 & 42 & 42  & 44\\ 
        & $10^{-3}$ & $10^{4}$  & 58 & 56 &  58   & 58\\
         & $10^{-6}$ & $1$   & 40 & 42 &  44   &44\\ 
        & $10^{-6}$ & $10^{4}$   &  56& 56 &  58  & 58\\
       \hline 
       GMRES & $10^{-3}$ & $1$   & 25 & 25 & 25  &25 \\ 
        & $10^{-3}$ & $10^{4}$   & 33 & 33 & 33   &33 \\ 
        & $10^{-6}$ & $1$   & 25 &25  &  25  & 25\\ 
        & $10^{-6}$ & $10^{4}$   &33  & 33 & 33  & 33\\ \hline
    \end{tabular}
    \label{Poro-3D}
\end{table}

\subsection{A linear poroelastic simulation of the spinal cord}

Now we consider the simulation of a real-world application: cross-sections of the spinal cord.
Anatomically, the spinal cord consists of gray and white matter, with the gray matter forming an H-shaped (or butterfly-shaped) structure in cross-section, surrounded by the white matter. 
As shown in Fig.~\ref{fig:spinal_DTI}(a), the inner marked in gray represents the gray matter,
surrounded by the white matter (in the darker area). There is the pia mater adheres to the white matter.
It is known (e.g., see \cite{khan,Stoverud25042016}) that the spinal cord can be modeled
as a poroelastic medium, with the flow and deformation in the spinal cord being
governed by the linear poroelasticity equations.
These modeling and studies are of importance to the understanding and treatments
of spinal cord injuries; e.g., see \cite{YU2020104044}.

Following \cite{khan,Stoverud25042016} we consider a scenario with the parameter values
listed in Table~\ref{tab:value_parameter}. Notice that Young's modulus
and the permeability have different values in the gray and white matter and pia mater.
The system has an antero-posterior diameter of 0.9\,cm and a transverse diameter of 1.3\,cm.
Fig.~\ref{fig:spinal_DTI}(b) shows a mesh used for the problem, which was smoothed
using a moving mesh method through the {Matlab} package \texttt{MMPDElab} \cite{Huang_MMPDE}.
The domain boundary is divided into three segments as
\begin{itemize}
    \item Boundary 1: The anterior part of the pia mater, with a length of 0.4\,cm
    (on top of the mesh, marked in red color);
    \item Boundary 2: The posterior part of the pia mater, with a length of 0.4\,cm
    (at bottom of the mesh, marked in blue color);
    \item Boundary 3: The remaining boundary (marked in black color).
\end{itemize}
Motivated by \cite{khan,Meddahi_SISC_2023, Stoverud25042016}, we impose mixed boundary conditions
on the poroelastic system as follows.
\begin{itemize}
    \item Displacement: 
    \begin{itemize}
        \item Boundary 1: 
        Time-dependent traction
        \[
        \V{t}_N = \Big(0,-9000\,(t/0.1)^{0.5} e^{-5 t+0.5}\Big)^T,
        \]
        with the maximum value of $9000$, as shown in Fig.~\ref{fig_traction_2}. 
        \item Boundary 2: Homogeneous Dirichlet boundary condition, $\V{u}_D = \V{0}$.
        \item Boundary 3: Traction-free boundary condition, $\V{t}_N = \V{0}$.
    \end{itemize}
    \item Pressure:
    \begin{itemize}
        \item Boundary 1: Time-dependent fluid pressure, $p_D = 9000\,(t/0.1)^{0.5} e^{-5 t+0.5} \, (\mathrm{dyne/cm}^2)$.
        \item Boundary 2: Traction-free for pressure, $p_N = 0$.
        \item Boundary 3: Homogeneous Dirichlet boundary condition, $p_D = 0$.
    \end{itemize}
\end{itemize}
Note that the boundary condition $p_D(t)$ imposed on the top of the pia mater
is a rough fitting to the actual data shown in \cite[Figure 4]{Stoverud25042016},
which was measured \textit{in vivo} from a Chiari patient at Oslo University Hospital.

The above boundary conditions simulate a situation with an indentation applied at the top of the pia mater,
a fixed bottom, and zero traction on the lateral boundaries. Under these boundary conditions,
the system will experience vertical compression and lateral expansion. 
This leads to more horizontal deformation within the tissue. 
At the same time, the fluid is expected to flow from higher pressure regions
to lower pressure regions and exit through the bottom of the pia mater,
while zero-pressure conditions are imposed on the two lateral sides.

We take $\V{f} = \V{0}$ (body force), $s = 0$ (fluid source), $\alpha = 1$, and $c_0 = 10^{-6}$,
and use the parameter values in Table~\ref{tab:value_parameter}.
Since the Lam\'e constant $\lambda$ increases with Young's modulus and leads to nearly incompressible of the elastic material, the pia mater is significantly stiffer than both white and gray matter.
As a result, the simulation of this problem involves a near-incompressibility condition.
Moreover, the discontinuities in Young's modulus and permeability can also slow down
the convergence of iterative solution for the entire system.
Thus, this application problem is also a good test example for the proposed
preconditioning strategies.

\begin{table}[htbp]
\centering
\begin{tabular}{|c|c|c|c|}
\hline
  & Young's modulus ($\mathrm{dyne/cm}^2$)& Poisson's ratio & Permeability $\kappa$ ($\mathrm{cm^4/(dyne\cdot s)}$) \\
\hline
pia mater & 2.3$\times 10^{7}$ & 0.479 & $\frac{3}{7}\times 10^{-8}$ \\
white matter & 5$\times 10^{4}$  & 0.479 & $2\times 10^{-8}$ \\
gray matter & 5$\times 10^{4}$ & 0.479 & $2\times 10^{-9}$ \\
\hline
\end{tabular}
\caption{Values of parameters from \cite{Stoverud25042016}.}
\label{tab:value_parameter}
\end{table}

Fig.~\ref{fig_spinal} presents the displacement magnitude and pressure at various time instants
$t = 0.035$, $0.1,$ and $0.5$ seconds.
We take $\Delta t = 0.005$ on a mesh with 22896 elements. 
The linear interpolation (Matlab function {\em shading interp}) is applied to smoothly visualize the field over the domain.
As shown in Fig.~\ref{fig_traction_2}, the indentation and fluid pressure imposed on the top of the pia mater increase until $t = 0.1$\,s and decrease afterwards.
Accordingly, greater deformation is expected at $t = 0.1$\,s and less at later times.
Moreover, the fluid pressure reaches a maximum value of 9000 $(\mathrm{dyne/cm}^2)$ at $t = 0.1$\,s.
These can be observed in Fig.~\ref{fig_spinal}.
More specifically, as shown in panels (a), (c), and (e),  the displacement magnitude increases
from $t = 0.035$\,s to 0.1\,s (reaching a maximum of approximately 0.07\,cm)
and then decreases due to the reduced indentation.
The posterior of the pia mater has zero displacement, indicating a fixed boundary,
while lateral deformation occurs due to the applied loading.
Numerical pressure results in panels (b), (d), and (f) imply that the pressure value imposed on the top of the pia mater are around 7000, 9000, and 3000 $(\mathrm{dyne/cm}^2)$ at $t = 0.035$\,s, $0.1$\,s, and $0.5$\,s,
respectively.
These results are consistent with those shown in Fig.~\ref{fig_traction_2}.
The pressure at the bottom increases over time, indicating that fluid gradually
flows throughout the domain.

The number of iterations required by preconditioned MINRES and GMRES to reach a specified tolerance within a single time step is presented in Table~\ref{table-spinal}.
We consider time steps $\Delta t = 10^{-3}$ and $ 10^{-6}$. The results indicate that both MINRES and GMRES remain stable with respect to changes in time step size and mesh refinement.
It is worth noting that there are jumps in the Lamé constant and permeability across the three layers.
Although our analysis of MINRES and GMRES assumes constant parameters over the whole domain, the convergence of the iterative solvers remains robust even in the presence of parameter jumps.

\begin{figure}[htbp]
    \centering
    \subfigure[]{\includegraphics[height=2in]{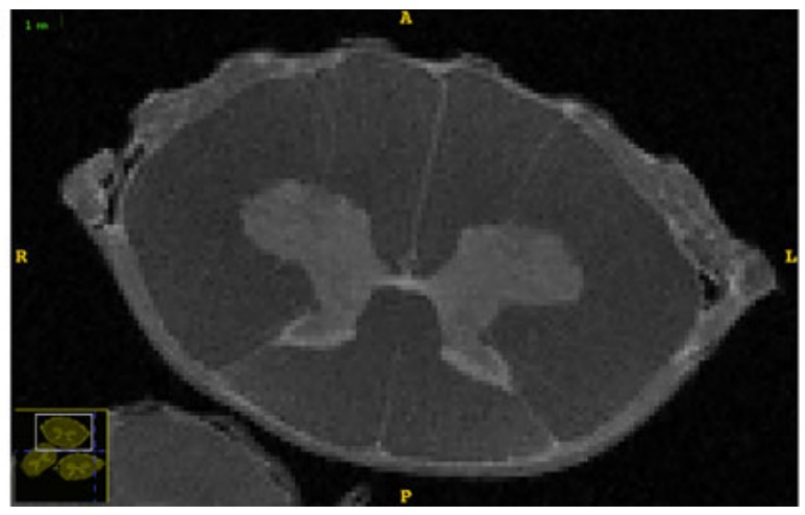}}
    % \hspace{0.05\linewidth}  
    % 
    \subfigure[]{\includegraphics[height=2in]{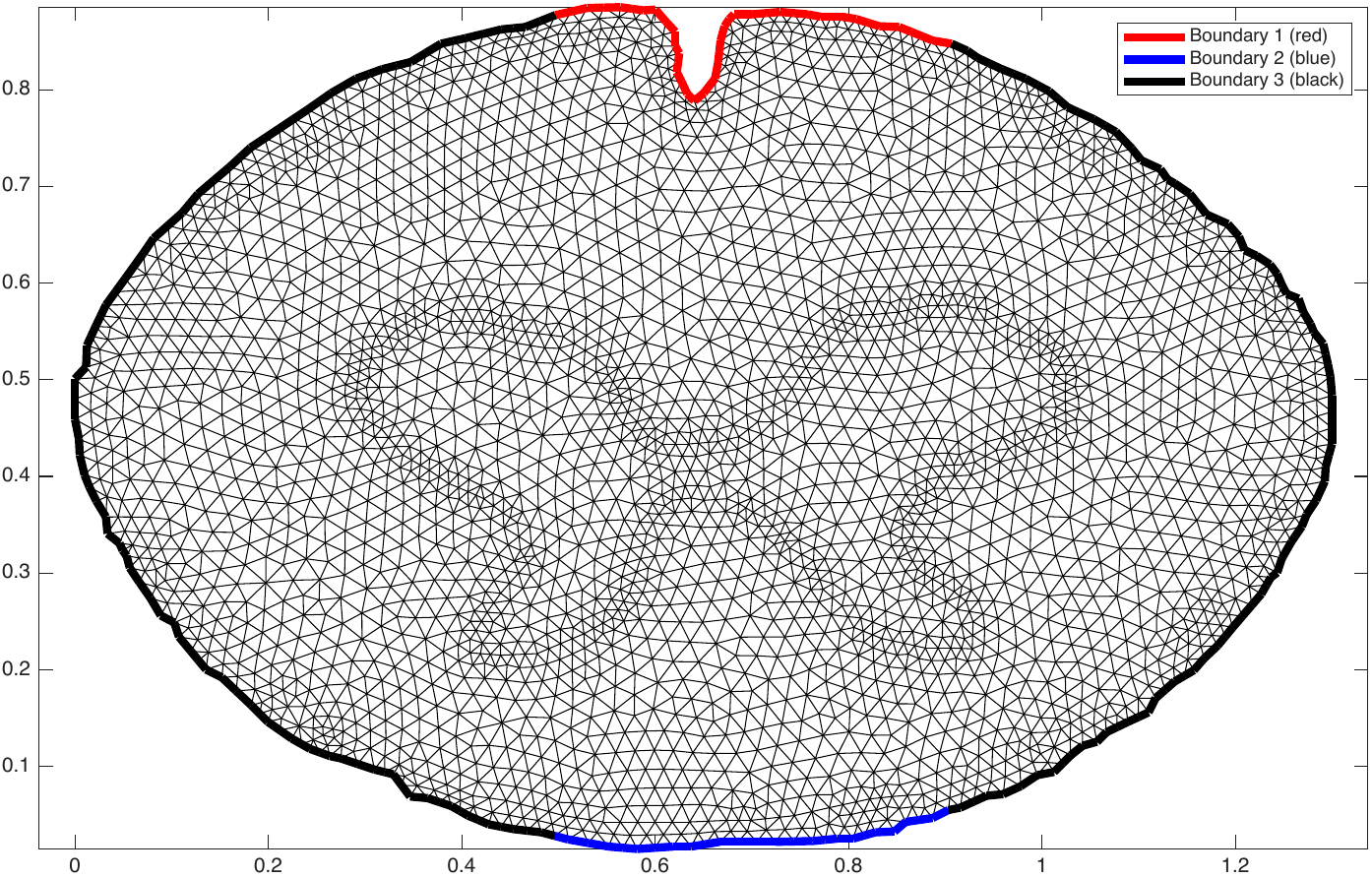}}
    \vspace{-0.1in}
     \caption{(a) High-resolution DTIs from a sheep cadaveric spinal cord \cite{Stoverud25042016}.
     (b) A sample triangular mesh with 5724 elements for the simulation.}
    \label{fig:spinal_DTI}
\end{figure}

\begin{figure}[htbp]
  \centering
  \includegraphics[width=0.5\textwidth]{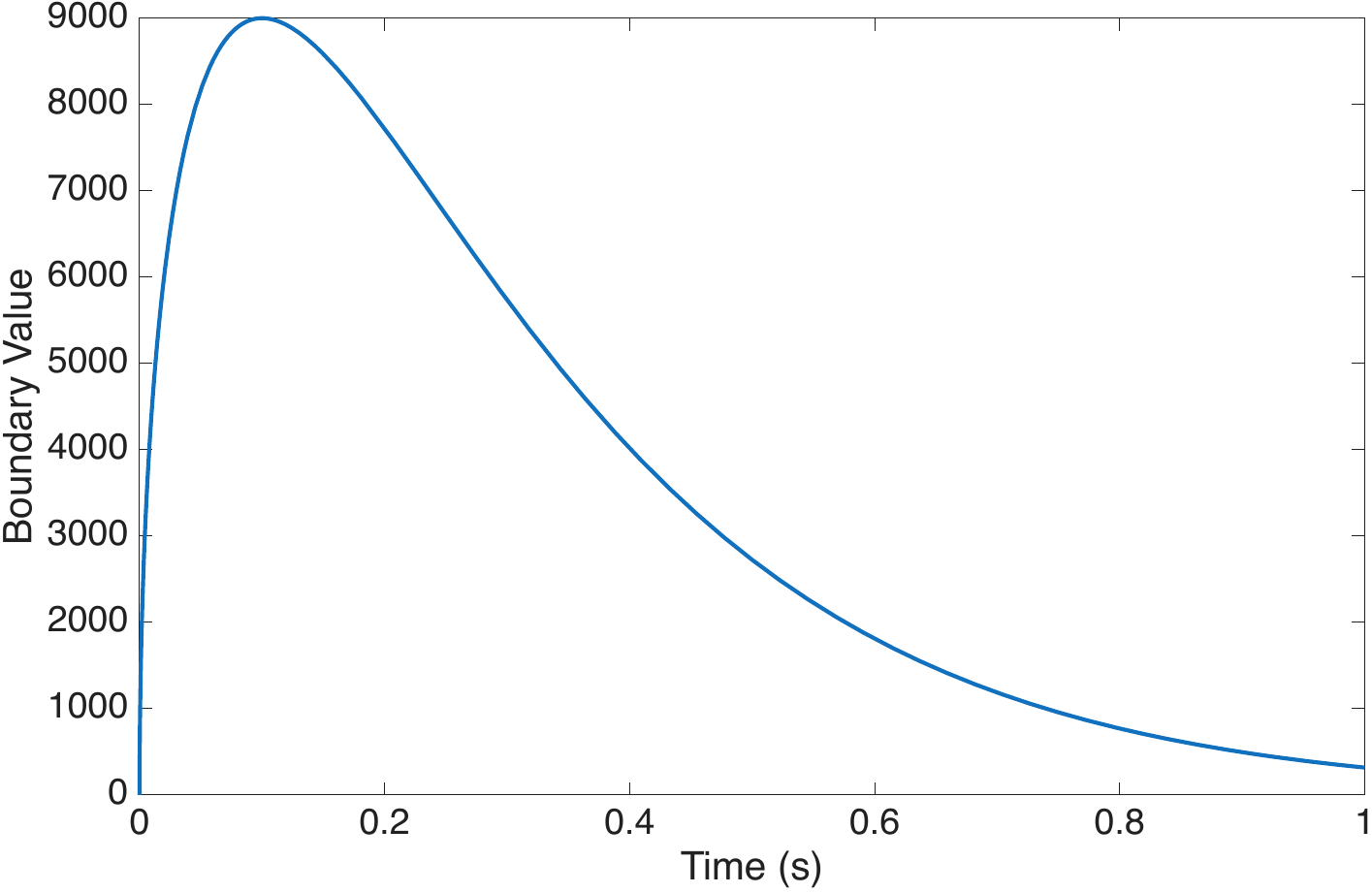}
  \vspace{-0.1in}
  \caption{A time-dependent downward traction and pressure imposed on the top of the pia mater.}
  \label{fig_traction_2}
\end{figure}
  
\begin{figure}[htbp]
  \centering
    \subfigure[Displacement magnitude at $t = 0.035$\,s]{
\includegraphics[width=0.45\textwidth]{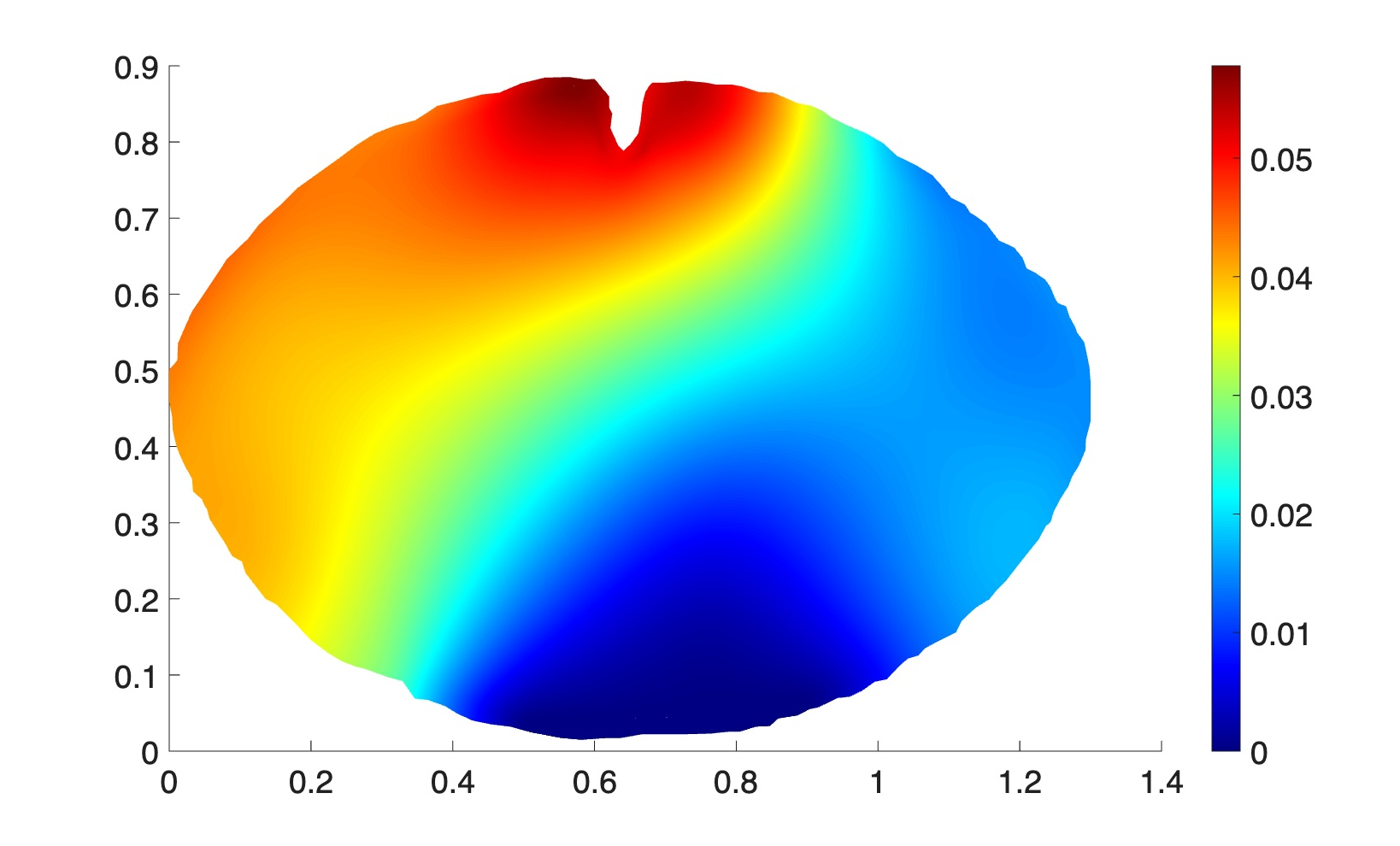}
  }
\subfigure[Pressure at $t = 0.035$\,s]{
\includegraphics[width=0.45\textwidth]{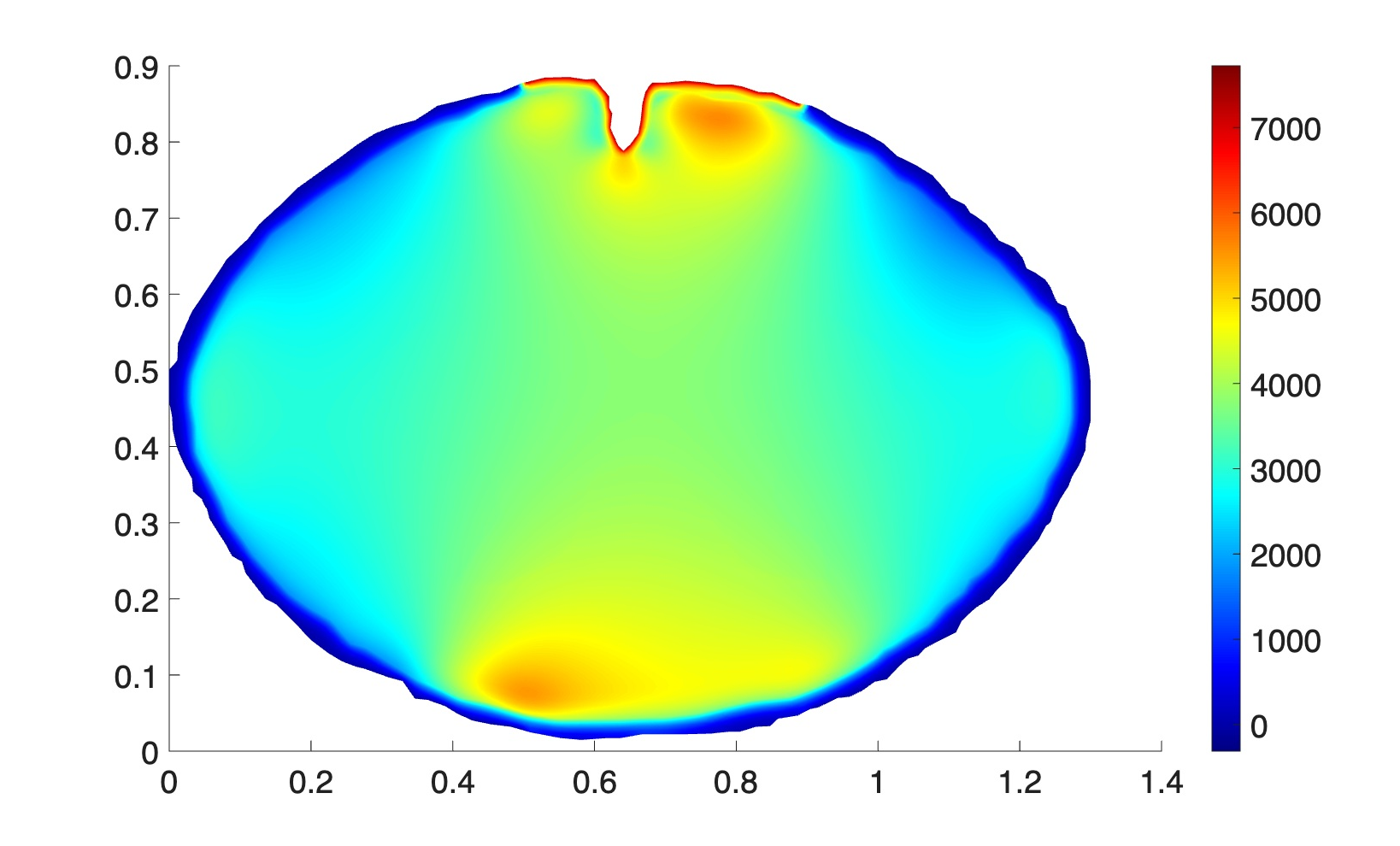}
  }
  \\
\subfigure[Displacement magnitude at $t = 0.1$\,s]{
\includegraphics[width=0.45\textwidth]{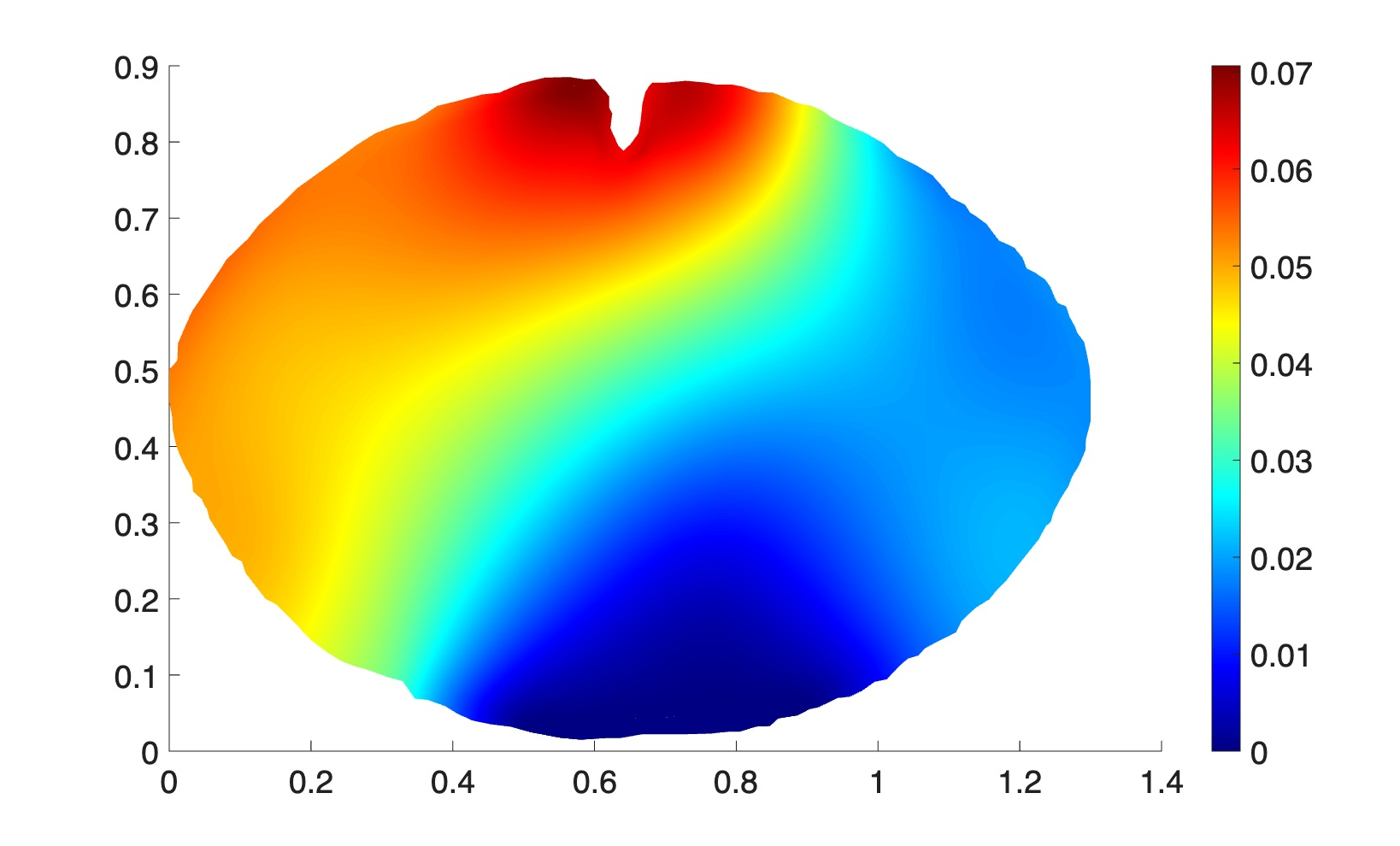}
  }
\subfigure[Pressure at $t = 0.1$\,s]{
\includegraphics[width=0.45\textwidth]{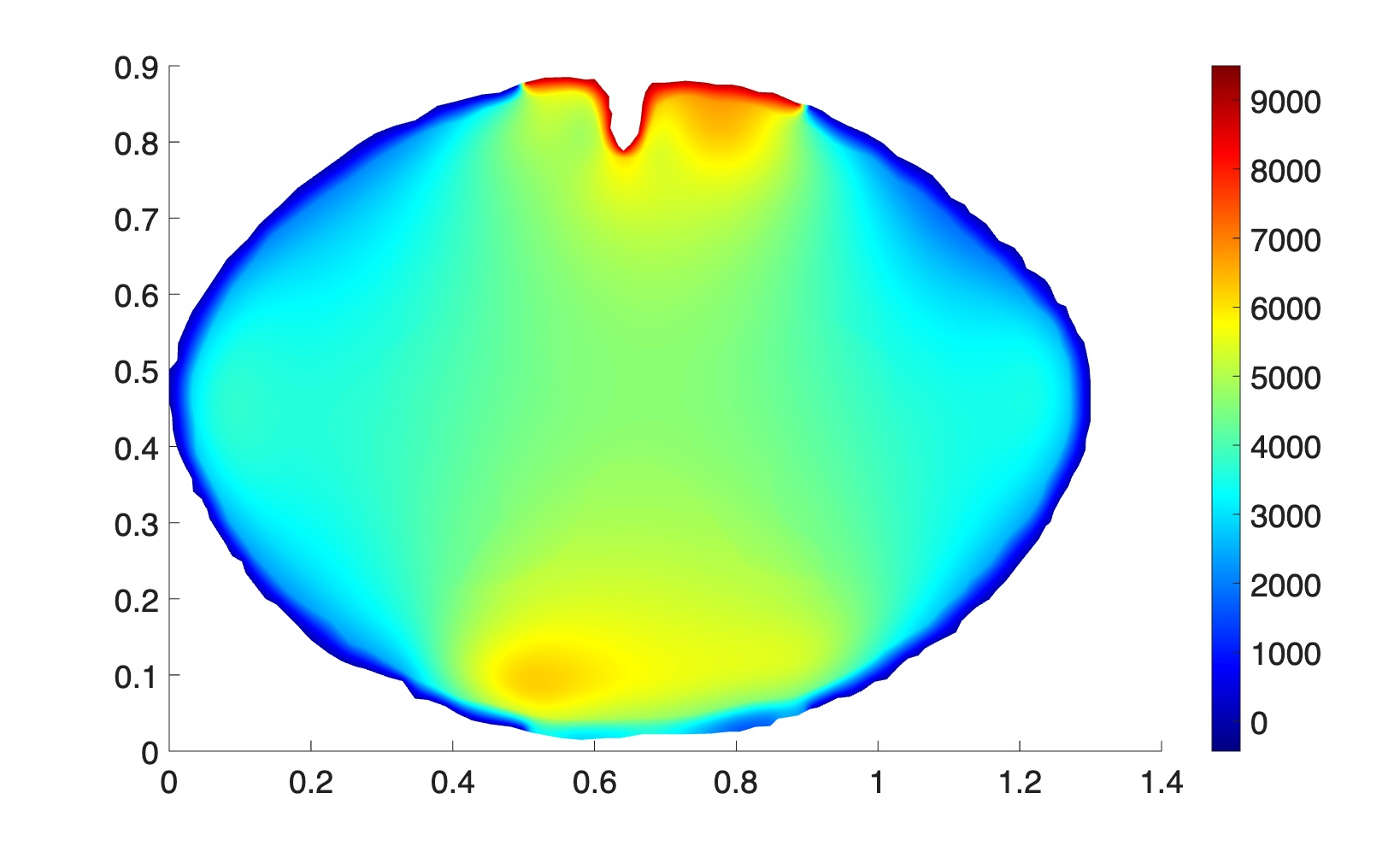}
  }
  \\
\subfigure[Displacement magnitude at $t = 0.5$\,s]{
\includegraphics[width=0.45\textwidth]{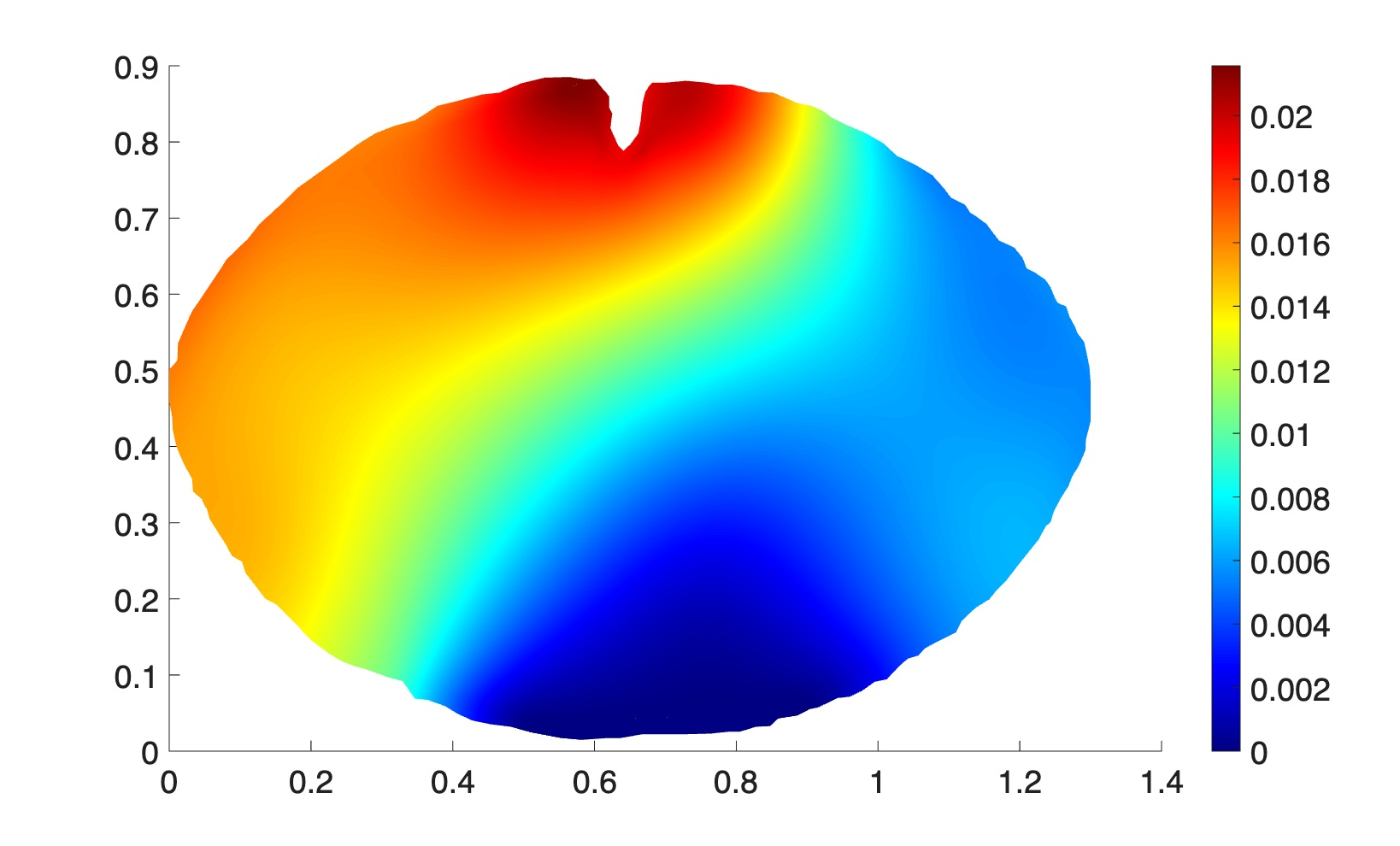}
  }
\subfigure[Pressure at $t = 0.5$\,s]{
\includegraphics[width=0.45\textwidth]{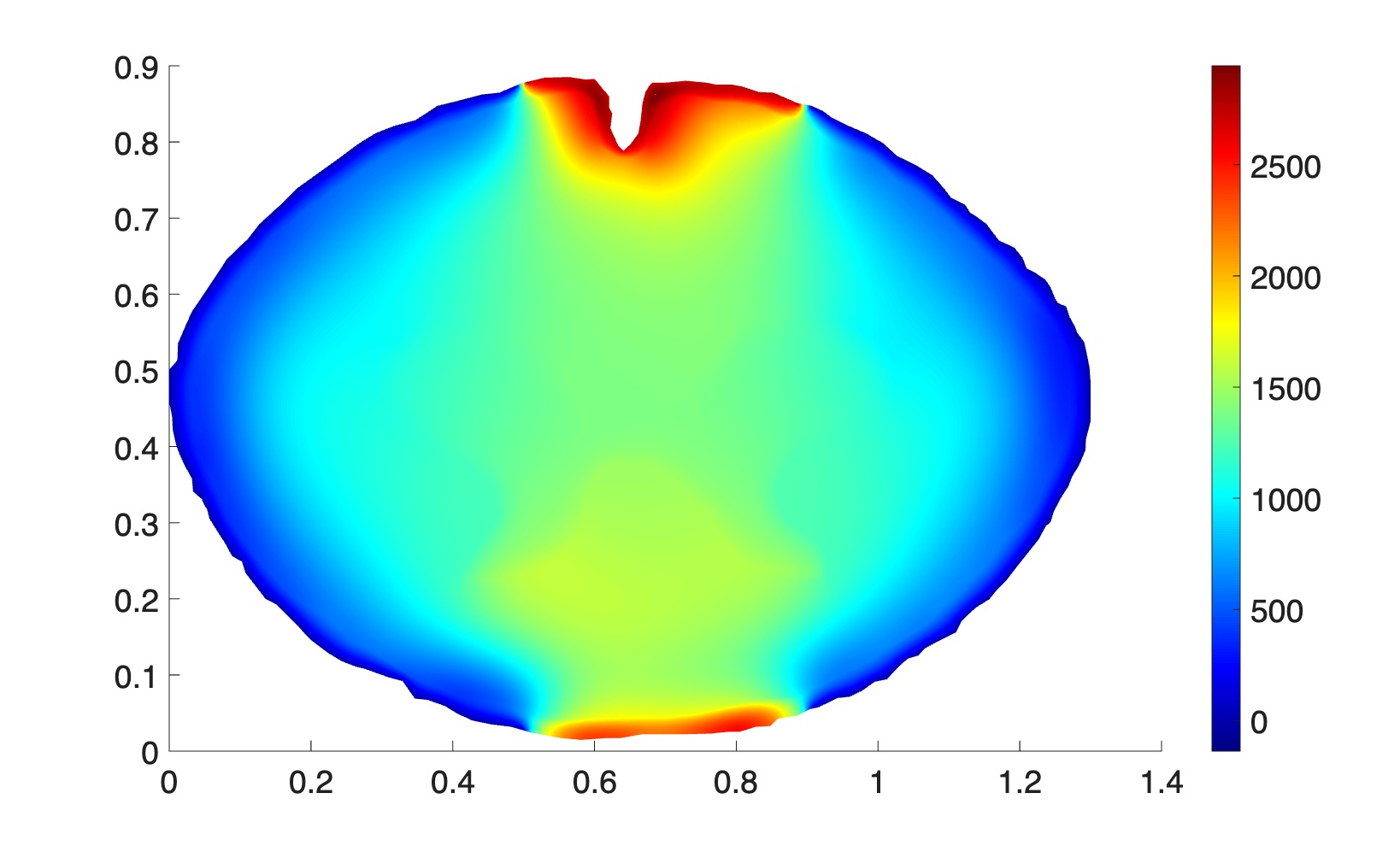}
  }  
  \caption{Numerical results for the system composed of pia mater and white and gray matter
  at different time steps.}
  \label{fig_spinal}
\end{figure}

% This is result for E1 = 2.3*1e2, E2 = 5e5, E3 = 5e5
% 
% \begin{table}[htb!]
% \centering
% \caption{Spinal cord simulation: The number of MINRES and GMRES iterations required to reach convergence for preconditioned systems, $\Delta t = 10^{-3} \text{ and } 10^{-6}$.}
% \begin{tabular}{|c|c|c|c|c|}
% \hline
% \multirow{2}{*}{} & \multirow{2}{*}{$\Delta t$} & \multicolumn{3}{c|}{$N$} \\ \cline{3-5}
% &  & 5724 & 22896 & 91584 \\ \hline
% MINRES   & $10^{-3}$ & 85  & 85   &  87   \\
% & $10^{-6}$ & 79  &  79  &  80   \\ \hline
% GMRES  & $10^{-3}$ & 38   &  38   &38 \\
% & $10^{-6}$  & 37   &  37   & 37  \\ \hline
% \end{tabular}
% \label{table-spinal}
% \end{table}

\begin{table}[htb!]
\centering
\caption{Spinal cord simulation: The number of MINRES and GMRES iterations required to reach convergence for preconditioned systems, $\Delta t = 10^{-3} \text{ and } 10^{-6}$.}
\begin{tabular}{|c|c|c|c|c|}
\hline
\multirow{2}{*}{} & \multirow{2}{*}{$\Delta t$} & \multicolumn{3}{c|}{$N$} \\ \cline{3-5}
&  & 5724 & 22896 & 91584 \\ \hline
MINRES   & $10^{-3}$ &  98 & 107   &  107   \\
& $10^{-6}$ & 90  &   91 &   91  \\ \hline
GMRES  & $10^{-3}$ &  51  &   51  &51 \\
& $10^{-6}$  & 51   &  51   &  50 \\ \hline
\end{tabular}
\label{table-spinal}
\end{table}

%%%%%%%%%%%%%%%%%%%%%%%%%%%%%%%%%%%%%%%%%%%%%%%%%%%%%%%%%%%%%%%
%%%%%%%%%%%%%%%%%%%%%%%%%%%%%%%%%%%%%%%%%%%%%%%%%%%%%%%%%%%%%%%
\newpage
\section{Conclusions}
\label{SEC:conclusions}

In the previous sections, we analyzed the convergence behavior of MINRES and GMRES,
preconditioned with inexact block diagonal and block triangular Schur complement preconditioners, respectively,
for solving the linear poroelasticity problem. 
The problem is discretized in a hybrid approach: using Bernardi–Raugel elements for the solid displacement,
lowest-order weak Galerkin elements for the pressure, and the implicit Euler method for time integration.
When pure Dirichlet boundary conditions are applied to the displacement,
the leading block $\epsilon A_1 + A_0$ becomes nearly singular in the locking regime,
posing significant challenges for iterative solvers. In contrast, when mixed boundary conditions are applied,
this block remains nonsingular, ensuring that the full system is also nonsingular.

To address the difficulties arising under pure Dirichlet conditions, we introduced
a numerical pressure variable $\V{z}_h$, leading to a three-field formulation described in Section~\ref{sec::DBC}.
This formulation is regularized by adding $-\rho \V{w}\V{w}^T \V{z}_h = \V{0}$ into the third block equation,
where $\V{w}$ is defined in (\ref{w-1}). This regularization preserves the solution and
ensures that the eigenvalues of the preconditioned Schur complement of the regularized system
are bounded above and below by positive constants, as shown in Lemma~\ref{lem:eigen_bound_3field_tilde},
assuming a quasi-uniform mesh.

For the regularized system, we developed inexact block diagonal and triangular Schur complement preconditioners. Convergence bounds for the corresponding MINRES and GMRES methods, presented in Propositions~\ref{pro:MINRES_conv-DBC} and~\ref{pro:gmres_conv_DBC}, demonstrate that convergence is essentially independent of the mesh size $h$ and the locking parameter $\lambda$.

When mixed boundary conditions are used, the algebraic system remains nonsingular even in the locking regime, eliminating the need for regularization. In this case, MINRES and GMRES, preconditioned as before,
were analyzed in Section~\ref{sec::NBC}, with convergence results provided in Propositions~\ref{pro:MINRES_conv-NBC} and~\ref{pro:gmres_tri_poro-NBC}. These results again indicate mesh- and parameter-independent convergence.

Section~\ref{SEC:numerical} presents numerical experiments for linear poroelasticity in both two and three dimensions. The results confirm the benefits of regularization under pure Dirichlet conditions and
the robustness of the preconditioners with respect to mesh size and the locking parameter
in both boundary condition scenarios.
Additionally, we applied the method to a spinal cord simulation involving discontinuous material parameters.
Even in this complex, real-world setting, both solvers demonstrated robustness and efficiency.

%%%%%%%%%%%%%
\section*{Acknowledgments}

W.~Huang was supported in part by the Simons Foundation grant MPS-TSM-00002397.

%---------------------------
% \bibliographystyle{abbrv}
% \bibliography{bibfile}

\end{document}